\RequirePackage{fix-cm}
\documentclass[smallextended]{svjour3}       
\smartqed  
\usepackage{graphicx}
\usepackage{amsmath}
\usepackage{amssymb}
\usepackage{amsfonts}

\usepackage{amsthm}
\usepackage{graphicx}
\usepackage{epstopdf}
\usepackage{enumerate}
\usepackage[linesnumbered]{algorithm2e}
\usepackage[colorinlistoftodos]{todonotes}
\usepackage{soul}
\usepackage[title]{appendix}
\usepackage{xcolor}
\usepackage[nocompress]{cite}

\usepackage{graphics}
\usepackage{array}

\usepackage[a4paper,top=3cm,bottom=2cm,left=2.5cm,right=2.5cm,marginparwidth=2cm]{geometry}

\newcounter{claimcounter}
\newcounter{claimproofcounter}

\newenvironment{claime}[1]{\stepcounter{claimcounter}\par\noindent\underline{Claim \theclaimcounter:}\space#1}{}
\newenvironment{claimproof}[1]{\stepcounter{claimproofcounter}\par\noindent\underline{Proof \theclaimproofcounter:}\space#1}{\hfill~$\blacksquare$}

\newcommand{\cB}{\mathcal{B}}
\newcommand{\cC}{\mathcal{C}}
\newcommand{\cF}{\mathcal{F}}

\newcommand{\cM}{\mathcal{M}}
\newcommand{\cN}{\mathcal{N}}
\newcommand{\cS}{\mathcal{S}}

\newcommand{\cmark}{\checkmark}
\newcommand{\xmark}{\text{\sffamily X}}
\newcommand{\qmark}{\text{\sffamily ?}}

\newcommand{\leoR}[1]{#1}

\newcommand{\review}[1]{#1}
\newcommand{\reviewx}[1]{#1}

\definecolor{lightgreen}{RGB}{180,255,150}

\definecolor{lightblue}{RGB}{180,180,255}

\usepackage{hyperref}
\hypersetup{
    colorlinks,
    linkcolor={red!50!black},
    citecolor={blue!50!black},
    urlcolor={blue!80!black}
}
\begin{document}

\title{Reconstructing Tree-Child Networks from Reticulate-Edge-Deleted Subnetworks
}


\author{Yukihiro Murakami   \and
        Leo van Iersel      \and
        Remie Janssen       \and
        Mark Jones          \and
        Vincent Moulton
}


\institute{
            Yukihiro Murakami   \at
              Delft Institute of Applied Mathematics, Delft University of Technology, Van Mourik Broekmanweg 6, 2628 XE, Delft, The  Netherlands\\
              \email{yukimurakami07201994@gmail.com}
            \and
            Leo van Iersel      \at
              Delft Institute of Applied Mathematics, Delft University of Technology, Van Mourik Broekmanweg 6, 2628 XE, Delft, The  Netherlands\\
              \email{l.j.j.v.iersel@gmail.com}
            \and
            Remie Janssen       \at
              Delft Institute of Applied Mathematics, Delft University of Technology, Van Mourik Broekmanweg 6, 2628 XE, Delft, The  Netherlands\\
              \email{remiejanssen@gmail.com}
            \and
            Mark Jones          \at
              Delft Institute of Applied Mathematics, Delft University of Technology, Van Mourik Broekmanweg 6, 2628 XE, Delft, The  Netherlands\\
              \email{markelliotlloyd@gmail.com}
            \and
            Vincent Moulton     \at
              School of Computing Sciences, University of East Anglia, NR4 7TJ, Norwich, United Kingdom\\
              \email{V.Moulton@uea.ac.uk}
}

\date{Received: date / Accepted: date}

\maketitle

\begin{abstract}
 Network reconstruction lies at the heart of phylogenetic research.
 Two well studied classes of phylogenetic networks include tree-child networks and level-$k$ networks.
 In a tree-child network, every non-leaf node has a child that is a tree node or a leaf.
 In a level-$k$ network, the maximum number of reticulations contained in a biconnected component is~$k$.
 Here, we show that level-$k$ tree-child networks are encoded by their reticulate-edge-deleted subnetworks, which are subnetworks obtained by deleting a single reticulation edge, if~$k\geq 2$.
 Following this, we provide a polynomial-time algorithm for uniquely reconstructing such networks from their reticulate-edge-deleted subnetworks. Moreover, we show that this can even be done when considering subnetworks obtained by deleting one reticulation edge from each biconnected component with~$k$ reticulations.
\end{abstract}

\section{Introduction}

Phylogenetic trees are instrumental in representing the evolutionary history of a set of species~$X$.
Leaves (extant species) are bijectively labelled by~$X$, and speciation events are depicted by internal nodes (non-extant species).
Though powerful in their own right, phylogenetic trees are limited by their inability to display complex evolutionary events such as horizontal gene transfers, hybridizations, and recombinations~\cite{sneath1975cladistic}.
For such reticulate (non-treelike) events, there has been increased interest in employing \emph{phylogenetic networks} instead, which are generalizations of phylogenetic trees to directed acyclic graphs~\cite{morrison2005networks,huson2010phylogenetic}.





In recent years, heavy focus has been cast upon the reconstruction of phylogenetic networks.
Many existing methods of tree reconstruction such as maximum parsimony, maximum likelihood, and distance-based methods have been adapted to network reconstruction~\cite{hein1990reconstructing,von1993network,strimmer2000likelihood,jin2006maximum,bordewich2018recovering,huson2010phylogenetic}.
In this paper, we tackle the reconstruction problem through a building block approach.
Building blocks are generally some class of subnetworks, e.g. binets~\cite{van2017binets}, trinets~\cite{huber2013encoding}, or trees, used to infer the original network.
A potential problem here is that there could be more than one network with the same building blocks.
When considering trees as building blocks, Pardi and Scornavacca somewhat resolved this distinguishability issue by considering `canonical forms' of networks; however the problem still persists in general~\cite{pardi2015reconstructible}.
Therefore the goal in any building block approach is to see if it \emph{encodes} the network.
We say that a network is \emph{encoded} by a certain building block if given two networks containing the same set of this building block, the networks are isomorphic.

It has been shown by Huber et al. in~\cite{huber2014much} that there exist networks 
which are not encoded by all subnetworks (called \emph{subnets}) induced on proper subsets of the taxa.
This is not to say that subnets do not encode many networks; in fact, it has been shown time and time again that considering topologically restricted classes of networks can help bypass this complication~\cite{willson2011regular,van2014trinets,gambette2017challenge,van2017binets}.
Two of the more prominent network classes are the \emph{tree-child networks}~\cite{cardona2009comparison} and the \emph{level-$k$ networks}~\cite{jansson2006inferring}.
In a tree-child network, every non-leaf node has a child that is a tree node (nodes with indegree-1 and outdegree-2) or a leaf (nodes with indegree-1 and outdegree-0).
In a level-$k$ network, the maximum number of reticulations (nodes with indegree-2 and outdegree-1) in a biconnected component (blob) is~$k$ (see Figure~\ref{fig:PhylNetwork} for an example of a level-$4$ tree-child network).

\begin{figure}
 \centering
 \includegraphics[scale = 0.7]{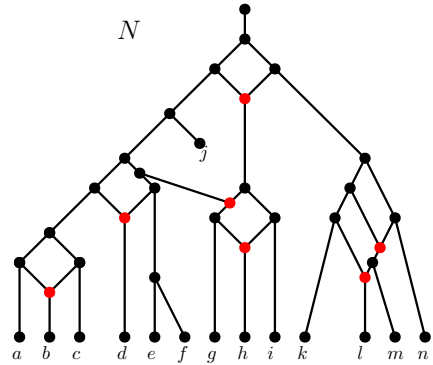}
 \caption{A level-4 tree-child network~$N$ on the set of species~$X = \{a,\dots, n\}$.
 Though~$N$ is a directed acyclic graph, the edge directions are omitted to avoid cluttering. 
 The arcs are directed downwards.
 The leaf pair~$\{e,f\}$ is a \emph{cherry} since they share a common parent.
 The leaf pair~$\{a,b\}$ is a \emph{reticulated cherry} since the parent of~$a$ is also the parent of the parent of~$b$.}
 \label{fig:PhylNetwork}
\end{figure}

In this paper, we show that binary level-$k$ tree-child networks, where~$k\geq2$ are encoded by \emph{reticulate-edge-deleted subnetworks}, which are subnetworks obtained by deleting a single reticulation edge.
In fact, we prove an even stronger result that this network class is encoded by its \emph{Maximum Lower-Level Subnetworks (MLLSs)}, the subnetworks obtained by deleting a reticulation edge from every level-$k$ biconnected component.
We do so by exploiting the fact that tree-child networks contain either a cherry or a reticulated cherry~\cite{bordewich2016determining}.
Cherries need not be reconstructed, since they stay intact in every MLLS; therefore we focus on reconstructing reticulated cherries and show that they are uniquely reconstructible through an exhaustive case study.
In proving this result, we explore `blob trees', an underlying tree of a network, introduced initially by Gusfield and Bansal~\cite{gusfield2005fundamental}.
These labelled trees are obtained from networks by collapsing every biconnected component to a single node, labelling the node by its set of leaf-descendants, and removing the leaves.
In this paper we introduce the class of \emph{valid} networks; for such a class, we show that we can reconstruct the blob tree of the original network from the blob trees of all MLLSs (Theorem~\ref{thm:BloBTreeReconstruction}).
The class of tree-child networks are contained within the class of valid networks, and therefore the result also follows for tree-child networks.

In related literature, it has been shown that tree-child networks are encoded by trinets~\cite{van2014trinets} but not by trees (see Figure~\ref{fig:SubnetReconButNotLvlRecon}).
Gambette et al.~\cite{gambette2017challenge} showed that level-1 networks (which are necessarily tree-child) with girth (shortest cycle in underlying graph) at least 5 are reconstructible from their triplets.
The triplets are phylogenetic trees on~3 leaves; as the set of triplets can be computed from the set of all displayed trees of the network, level-1 networks of girth at least 5 are encoded by trees and therefore by their MLLSs.
\reviewx{Others have also constructed level-1 networks (also called gt-networks) from trees.
Nakhleh et al. showed that it was possible to find a level-1 network with the minimum number of reticulations that displays an input of two binary trees in polynomial time if such a network exists~\cite{nakhleh2005reconstructing}.
In the same paper, they also considered the following problem, which we restate using our notation.
Given an input of two nonbinary trees, find a level-$1$ network~$N$ with one reticulation such that~$N$ displays two MLLSs that are refinements of the two input nonbinary trees, if such a network exists.
Huynh et al.~\cite{huynh2005constructing} generalized this result by showing that one can find a level-$1$ network with the minimal number of reticulations for an input size of at least two nonbinary trees, if such a network exists.
They did not however consider whether the output network was unique, which is our focus for this paper.
In particular, we focus on showing that certain networks are uniquely defined by their MLLSs and that they can be uniquely reconstructed from them; this is fundamentally different from the problem of finding a most parsimonious network for a set of trees -- which, coincidentally, is a subset of MLLSs for a level-$1$ network.
}

The paper is organized as follows.
In the next section we define essential terms relevant to this paper, including \emph{MLLSs} and the notion of \emph{encoding / reconstructibility}.
Section~\ref{sec:BlobTree} presents the definitions and the key results on blob trees.
In Section~\ref{sec:LeafPairAnalysis} we investigate the possible topologies for each leaf pair.
Per our definition, there are~5 possibilities for each leaf pair up to isomorphism, and we develop a method for reconstructing a blob containing a particular leaf pair topology.
In Section~\ref{sec:TCReconstruction}, we show our main result for this paper, that binary level-$k\geq2$ tree-child networks are reconstructible from their MLLSs (Theorem~\ref{thm:TCNetworksReconstructible}).
A polynomial-time (in the size of the leaf set and the MLLS set) algorithm for reconstructing tree-child networks from their MLLSs follows naturally from our proof, and we present this in Section~\ref{sec:Algorithm}.
In the last section, we conclude with some discussion of potential future directions.


\section{Preliminaries and Definitions}\label{sec:Preliminaries}

\begin{definition}\label{def:PhylNetwork}
 Let~$X$ be a non-empty finite set.
 A \emph{rooted binary phylogenetic network}~$N$ on~$X$ is a \emph{directed acyclic graph} (a directed graph with no directed cycles) \review{in which every node is \leoR{in} one of the following \leoR{categories}:}
 \begin{enumerate}
  \item one node of indegree-0 and outdegree-1 (the \emph{root});
  \item $|X|$ nodes of indegree-1 and outdegree-0 (\emph{leaf nodes} or \emph{leaves});
  \review{
  \item nodes of indegree-1 and outdegree-2 (\emph{tree nodes}); and
  \item nodes of indegree-2 and outdegree-1 (\emph{reticulations}). 
  }
 \end{enumerate}
 The leaves are bijectively labelled with label set~$X$, where the leaf set is sometimes denoted~$L(N)$.
\end{definition}

We will henceforth refer to rooted binary phylogenetic networks as \emph{networks}. The edges feeding into reticulations are called \emph{reticulation edges} and each non-reticulation edge is called a \emph{tree edge}.
 We write~$v\in N$ to denote that~$v$ is a node in~$N$.
 Given an edge~$(x,y)$ in~$N$, we say that~$x$ is a \emph{parent} of~$y$ and~$y$ is a \emph{child} of~$x$.
 A \emph{directed path} of length~$n$ from~$x$ to~$y$ is a sequence of edges~$(v_0,v_1), \dots, (v_{n-1},v_n)$ such that~$x = v_0, y = v_n$, where~$v_i$ is a parent of~$v_{i+1}$ for~$i = 0,\dots,n-1$.
 The node~$x$ is an \emph{ancestor} of / \emph{above}~$y$, or~$y$ is a \emph{descendant} of / \emph{below}~$x$ if there is a directed path from~$x$ to~$y$ in~$N$.
 Two nodes are \emph{incomparable} if neither nodes are above the other.
 The network~$N$ is \emph{tree-child} if every non-leaf node in~$N$ is a parent of a tree node or a leaf. A \emph{tree path} is a directed path that contains no reticulations except possibly for its starting node. It is easy to see that, for each node~$v$ of a tree-child network, there exists a tree path to a leaf.
 
 Two networks~$N,N'$ on~$X$ are \emph{isomorphic} if there exists a bijection~$f$ between the vertices of~$N$ and the vertices of~$N'$ such that~$(u,v)$ is an edge of~$N$ if and only if~$(f(u),f(v))$ is an edge of~$N'$ and each leaf of~$N$ is mapped to a leaf of~$N'$ with the same label.

\begin{definition}
 \emph{Deleting} a node~$x$ from a network is the action of removing~$x$ and all of its incident edges from~$N$.
 \emph{Deleting} an edge~$(x,y)$ from a network is the action of removing~$(x,y)$ from~$N$.
\end{definition}

 \review{A \emph{cut-node} is a node of a network whose deletion disconnects the network.
 A \emph{cut-edge} is an edge of a network whose deletion disconnects the network.}
 A \emph{pendant subnetwork} of a network~$N$ is obtained by deleting a cut-edge~$(x,y)$ from~$N$ and taking the connected component containing~$y$. A \emph{pendant subtree} is a pendant subnetwork that is a tree.

\begin{definition}
A \emph{biconnected component} of a network~$N$ is a maximal subgraph with at least three nodes such that no node of the subgraph is a cut-node of the subgraph. A \emph{blob} is either a biconnected component or a tree node that is not in a biconnected component.
\end{definition} 

\review{We say~$N$ is a \emph{level-$k$ network}, denoted~$lvl(N) = k$, if the maximum number of reticulations contained in any biconnected component is~$k$~\cite{jansson2006inferring}.
 A level-0 network is a tree (a network with no reticulations).}
Since the level of a blob is the number of reticulations it contains, a tree node that is not in a biconnected component is a level-0 blob.

We say that a network~$N$ on~$X$ \emph{displays} a network~$N'$ on~$X$ 
if some subgraph \review{$N''$} of~$N$ is a subdivision of~$N'$ \review{(i.e., if~$N''$ can be obtained from~$N'$ by replacing directed edges by directed paths)}. 
An alternative view of when a network is displayed by another network is based on \emph{cleaning up} a directed acyclic graph. 

\begin{definition}
 \emph{Cleaning up} a directed acyclic graph 
 is the act of applying the following operations until none is applicable:
 \begin{enumerate}
  \item delete an unlabelled outdegree-0 node;
  \item \emph{suppress} an indegree-1 outdegree-1 node \review{(i.e., if $(u,v),(v,w)$ are edges in a graph where~$v$ is an indegree-1 outdegree-1 vertex, we \emph{suppress}~$v$ by deleting the node~$v$ and adding an edge~$(u,w)$.)};
  \item replace a pair of parallel edges by a single edge, i.e., delete one of the parallel edges and suppress both the parent node and the child node.
 \end{enumerate}
\end{definition}

Note that cleaning up a directed acyclic graph, obtained from a network on~$X$ by deleting, for each reticulation, at most one of the incoming reticulation edges,
returns a network on~$X$.

\begin{lemma}\label{lem:cleanup}
If a network~$N$ on~$X$ displays a network~$N'$ on~$X$, then we can obtain~$N'$ from~$N$ by deleting, for each reticulation, at most one of the two incoming reticulation edges, and subsequently cleaning up.
\end{lemma}

\begin{proof}
Since~$N$ displays~$N'$, some subgraph of~$N$ is a subdivision of~$N'$.
Because of this, there is an embedding of~$N'$ into~$N$ where the nodes and edges of~$N'$ are mapped to nodes and paths of~$N$, such that these paths are edge-disjoint. Without loss of generality, this embedding contains the root of~$N$.
For each reticulation of which exactly one incoming reticulation edge is used by the embedding, delete the other incoming reticulation edge, and subsequently clean up the directed acyclic graph.
We claim that all unused edges in the embedding have been removed in the resultant network~$M$.

Suppose not.
Then there exists an edge in~$M$ that is not used in the embedding of~$N'$ into~$N$. Consider a lowest such edge~$(x,y)$.

Node~$y$ cannot be a leaf of~$N$ because all leaves of~$N$ are in the embedding of~$N'$ into~$N$.

Now suppose that~$y$ is a tree node of~$N$. It is not possible that an outgoing edge of~$y$ is in the embedding, because the root of the embedding is the root of~$N$. Hence, the outgoing edges of~$y$ are not in the embedding. At least one of these outgoing edges of~$y$ is in~$M$ because otherwise~$y$ would have been deleted by cleaning up rule 1. Hence, at least one outgoing edge of~$y$ is in~$M$ but not in the embedding of~$N'$ into~$N$, contradicting the assumption that~$(x,y)$ is a lowest such edge.

Hence,~$y$ is a reticulation. If the other incoming edge of the reticulation is also not in the embedding, it follows similarly to the previous case that the outgoing edge of~$y$ is in~$M$ but not in the embedding, contradicting the assumption that~$(x,y)$ is a lowest such edge. Hence, exactly one incoming edge of~$y$ is used by the embedding. Therefore, the other incoming edge, $(x,y)$, has been deleted, contradicting the assumption that~$(x,y)$ is an edge of~$M$.


This implies that every edge in~$M$ is used in the embedding of~$N'$ into~$N$. Since in addition all indegree-1 outdegree-1 nodes have been suppressed by cleaning up rule 2, we have that~$M$ is~$N'$.
\end{proof}

Let~$\cN(N)$ denote the set of all networks on~$X$ that are displayed by a network~$N$ on~$X$, excluding~$N$ itself. The networks in~$\cN(N)$ are called the \emph{subnetworks} of~$N$.  
A class~$\cC$ of networks is called \emph{subnetwork-reconstructible} if for any two networks~$N,N'\in\cC$ with~$\cN(N)=\cN(N')$, we have that~$N$ and~$N'$ are isomorphic.


A related but subtly different notion is the following. Let~$N$ be a level-$k$ network.
Then~$\cN^{k-1}(N)$ denotes the set of subnetworks of~$N$ that \review{are} of level at most~$k-1$. The networks in~$\cN^{k-1}(N)$ are called the \emph{lower-level subnetworks} of~$N$.
Then a class~$\cC$ of networks is called \emph{level-reconstructible} if for any two networks~$N,N'\in\cC$ of level-$k$ and~$\cN^{k-1}(N)=\cN^{k-1}(N')$, we have that~$N$ and~$N'$ are isomorphic.
Note that if a network is level-reconstructible then it is subnetwork-reconstructible.
The converse is not true in general, and an example of this is shown in Figure~\ref{fig:SubnetReconButNotLvlRecon}.

\begin{figure}[h]
 \centering
 \includegraphics[scale = 0.5]{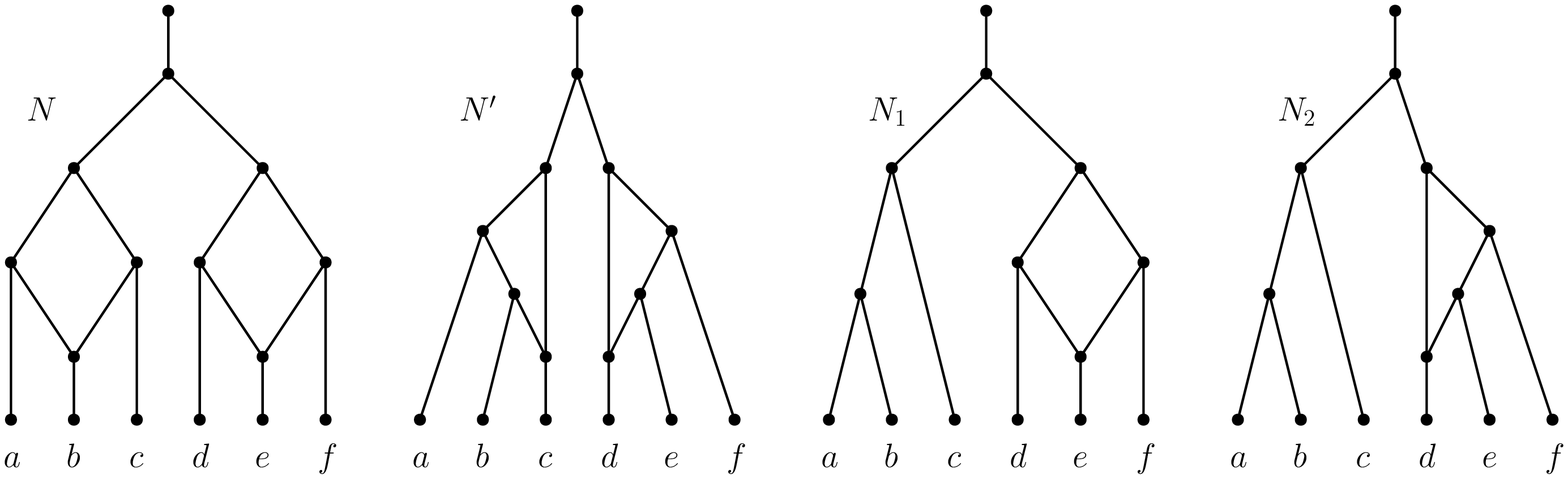}
 \caption{Networks~$N$ and~$N'$ are non-isomorphic but have the same lower-level subnetworks. 
 Hence, any class containing~$N$ and~$N'$ is not level-reconstructible.
 However, these networks have different subnetworks:~$N_1$ is a subnetwork of~$N$ but not of~$N'$;~$N_2$ is a subnetwork of~$N'$ but not of~$N$.
 So~$\{N,N'\}$ is subnetwork-reconstructible.}
 \label{fig:SubnetReconButNotLvlRecon}
\end{figure}

In this paper we prove a result that is stronger than level-reconstructibility.
We first define a type of reticulation edge deletion, and we introduce a corresponding subclass of networks.

\begin{definition}
A reticulation edge deletion is \emph{valid} if the resulting subnetwork, after cleaning up, contains exactly~2 nodes and~3 edges fewer than the original network, i.e., only the reticulation edge is deleted and its endpoints suppressed.
A reticulation edge deletion is \emph{invalid} otherwise.
Call a reticulation edge \emph{valid / invalid} if its deletion is valid / invalid.
\end{definition}

\begin{definition}\label{def:ValidNetwork}
 Networks are~\emph{valid} if all reticulation edges in the network are valid.
\end{definition}

An example of a valid reticulation edge is shown in Figure~\ref{fig:InvalidEdges}.

\begin{lemma}\label{lem:DeletingRetEdgeTCValid}
 All reticulation edges in a tree-child network are valid.
\end{lemma}
\begin{proof}
 Let~$N$ be a tree-child network and suppose for a contradiction that deleting some reticulation edge~$e = (u,v)$ is invalid.
 We note that~$v$ is a reticulation. As~$N$ is tree-child, we also have that~$u$ is a tree node.
 Therefore, after deleting~$e$,~$u$ and~$v$ will each be indegree-1 outdegree-1 nodes, and will be suppressed by cleaning up.
 This removes a total of 2 nodes and 3 edges. 
 Hence, to show that~$e$ is valid, it remains to show that no further cleaning up occurs after deleting~$e$ and suppressing~$u$ and~$v$.
 As all remaining nodes have the same indegree-and outdegree-as before, there are no unlabelled out-degree 0 nodes and no remaining indegree-1 outdegree-1 nodes. So we just need to show that deleting~$e$ creates no parallel edges.

 We split \review{the proof} into three sub-cases.
 First assume that suppressing~$u$ results in the creation of parallel edges.
 Then we must have that~$u$ is contained in a `triangle' with nodes~$x,y$ and edges~$(x,u),(x,y),(u,y)$.
 But then~$y$ is a reticulation, implying that~$u$ is the parent of two reticulations~$y$ and~$v$.
 Thus~$u$ has no child that is a tree node or a leaf, contradicting the tree-child property of~$N$.
 Next assume that suppressing~$v$ results in the creation of parallel edges.
 Then we must have that~$v$ is contained in a triangle with nodes~$x,y$ and edges~$(x,v),(x,y),(v,y)$.
 But then~$y$ is a reticulation, implying that~$v$ is the parent of a reticulation~$y$.
 Thus~$v$ has no child that is a tree node or a leaf, contradicting the tree-child property of~$N$.
 Finally assume that suppressing both~$u$ and~$v$ results in the creation of parallel edges.
 Then we must have that~$e$ formed the central edge of a `diamond' with nodes~$x,y$ and edges~$(x,u),(x,v),(u,v),(u,y),(v,y)$.
 However this cannot occur since the child of~$v$,~$y$, would be a reticulation, which again contradicts the tree-child property of~$N$.

 Therefore, every reticulation edge of a tree-child network is valid.
\end{proof}

\begin{figure}[h]
 \centering
 \includegraphics[width = \textwidth]{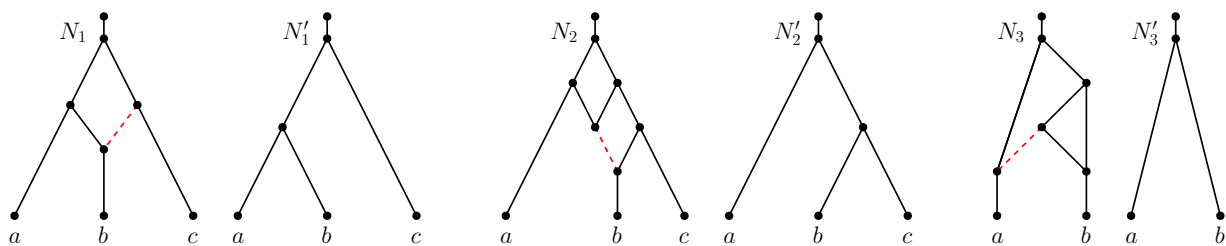}
 \caption{Three networks~$N_1, N_2, N_3$ with their respective maximum subnetworks~$N_1', N_2', N_3'$ obtained by deleting the red reticulation edge and subsequently cleaning up.
 The red reticulation edge in~$N_1$ is valid, however the red \review{dashed} reticulation edges in~$N_2$ and~$N_3$ are invalid.
 The subnetwork~$N_2'$ contains~4 fewer nodes and~6 fewer edges than~$N_2$, and~$N_3'$ contains~3 fewer nodes and~5 fewer edges than~$N_3$.
 }
 \label{fig:InvalidEdges}
\end{figure}
The above lemma does not hold for general networks (see Figure~\ref{fig:InvalidEdges}).
Intuitively, Lemma~\ref{lem:DeletingRetEdgeTCValid} states that removing any reticulation edge from a tree-child network is self-contained, and it does not affect any other reticulations within the network.
No additional information is `lost' when deleting valid reticulation edges.
In particular, Lemma~\ref{lem:DeletingRetEdgeTCValid} implies that tree-child networks are valid.

From here onwards, it is implicitly assumed that the network~$N'$ obtained by deleting some reticulation edges from~$N$ undergoes cleaning up.

\begin{definition}
 A \emph{maximum subnetwork} of a network~$N$ is a subnetwork obtained by a single reticulation edge deletion from~$N$.
\end{definition} 

\begin{lemma}\label{lem:RetEdgeDelNetTC}
 Every maximum subnetwork of a tree-child network is tree-child.
\end{lemma}

\begin{proof}
 Suppose that there exists a tree-child network~$N$ with a maximum subnetwork~$N'$ that is not tree-child.
 Then there exists a node~$t$ in~$N'$ such that all of its children are reticulations.
 Let~$(u,v)$ be the reticulation edge deleted from~$N$ to obtain~$N'$.
 Since~$t$ has a tree node as a child in~$N$, node~$u$ must be a child of~$t$ in~$N$. Hence,~$(t,u)$ and~$(u,r)$ are edges in~$N$, for some child~$r$ of~$t$ in~$N'$. 
 But then~$N$ is not tree-child as~$u$ is the parent of only reticulations~$v$ and~$r$, a contradiction 
 \review{(see Figure~\ref{fig:RetEdgeDelNetTC}).}
\end{proof}

\begin{figure}
    \centering
    \includegraphics[scale = 0.3]{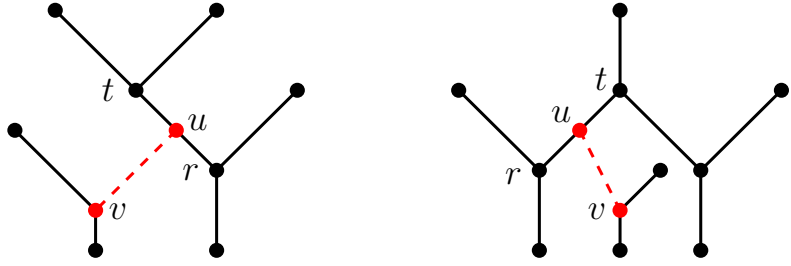}
    \caption{\review{Visual aid for the proof of Lemma~\ref{lem:RetEdgeDelNetTC}. The left case is when~$N'$ contains a reticulation~$t$ that is a parent of a reticulation~$r$. The right case is when a tree node~$t$ is a parent of two reticulations. In either case, the red dashed edge~$(u,v)$ must be inserted in these particular places to obtain~$N$, and in either case~$N$ is not tree-child.}}
    \label{fig:RetEdgeDelNetTC}
\end{figure}

\begin{definition}\label{def:maxllsubnetwork}
 \review{For~$k\geq1$}, a \emph{maximum lower-level subnetwork (MLLS)} of a level-$k$ network~$N$ is a subnetwork obtained by deleting exactly one valid reticulation edge from every level-$k$ blob in~$N$.
 Let~$\cN^{mlls}(N)$ denote the set of all MLLSs of~$N$.
 \end{definition}
 
Observe that, as long as~$\cN^{mlls}(N)$ is a non-empty set, it is equal to the set of all subnetworks of~$N$ with level at most~$k-1$ and a maximum number of edges.

By considering each reticulation edge deletion separately, it follows from Lemma~\ref{lem:RetEdgeDelNetTC} that the MLLSs of a tree-child network are tree-child.\\

A class~$\cC$ of networks is called \emph{MLLS-reconstructible} if for any two networks~$N,N'\in\cC$ with~$\cN^{mlls}(N) = \cN^{mlls}(N')$, we have that~$N$ and~$N'$ are isomorphic.
Because all MLLSs are lower level subnetworks of~$N$, we have~$\cN^{mlls}(N)\subseteq \cN^{k-1}(N)$.
Therefore, if a class of networks is MLLS-reconstructible, then it is level-reconstructible.
The converse also holds for valid networks.

\review{
\begin{lemma}\label{lem:ValidLLS=MLLS}
 Let~$N$ be a level-$k$ valid network. Then we may obtain~$\cN^{k-1}(N)$ from~$\cN^{mlls}(N)$.
\end{lemma}
\begin{proof}
 Let~$M\in\cN^{k-1}(N)$.
 As~$M$ is a lower-level subnetwork of~$N$, by Lemma~\ref{lem:cleanup},~$M$ must have been obtained from~$N$ by deleting at least one reticulation edge, say~$e_i$, from every level-$k$ blob, say~$B_i$, in~$N$, and deleting some reticulation edges from other blobs.
 By definition of MLLSs, there must exist an MLLS~$N'\in \cN^{mlls}(N)$ that was obtained from~$N$ by deleting~$e_i$ from~$B_i$.
 Then clearly, some subnetwork of~$N'$, obtained by deleting the rest of the reticulation edges, is~$M$.
 That is,~$\cN^{k-1}(N)$ is precisely the set of all subnetworks of the networks of~$\cN^{mlls}(N)$, and the networks of~$\cN^{mlls}(N)$ (i.e., $\cN^{k-1}(N) = \cN^{mlls}(N) \cup \bigcup_{M\in\cN^{mlls}(N)}\cN(M)$).
\end{proof}
\begin{corollary}
 Given a class~$\cC$ of networks containing only valid networks, if the class is level-reconstructible, then it is MLLS-reconstructible.
\end{corollary}
\begin{proof}
 Let~$N,N'\in\cC$ with~$\cN^{mlls}(N) = \cN^{mlls}(N')$.
 By Lemma~\ref{lem:ValidLLS=MLLS}, this implies that~$\cN^{k-1}(N) = \cN^{k-1}(N')$.
 As the class is level-reconstructible, we have that~$N$ and~$N'$ must be isomorphic.
\end{proof}
Note that this result does not hold in general, as networks may contain invalid reticulation edges that cannot be deleted to obtain an MLLS.
}

\begin{observation}\label{obs:reconstructibility}
Let~$\cC$ be a class of networks.
If~$\cC$ is MLLS-reconstructible, then~$\cC$ is also level-reconstructible.
If~$\cC$ is level-reconstructible, then~$\cC$ is also subnetwork-reconstructible.
\end{observation}

We will henceforth assume that all considered networks are binary tree-child networks on a non-empty set of taxa~$X$, unless stated otherwise.

\section{Blob Trees}\label{sec:BlobTree}

In this section we show how to reconstruct a \emph{blob tree}, the underlying tree of a network.
The tree has a similar construction as the `blobbed trees' in~\cite{gusfield2005fundamental} with further modifications.

\begin{definition}
 The \emph{blob tree} of a network~$N$, denoted~$BT(N)$, is the labelled tree obtained by applying the following:
 \begin{enumerate}
  \item contract every blob into a single node, and label each node, except for the root node, by the leaf-descendant set of the top node of the blob;
  \item delete all leaf nodes.
 \end{enumerate}
 We call the nodes in~$BT(N)$ \emph{blob nodes}.
\end{definition}

An example of a blob tree is illustrated in Figure~\ref{fig:BlobTreeExample}.

\begin{figure}[h]
 \centering
 \includegraphics[width = \textwidth]{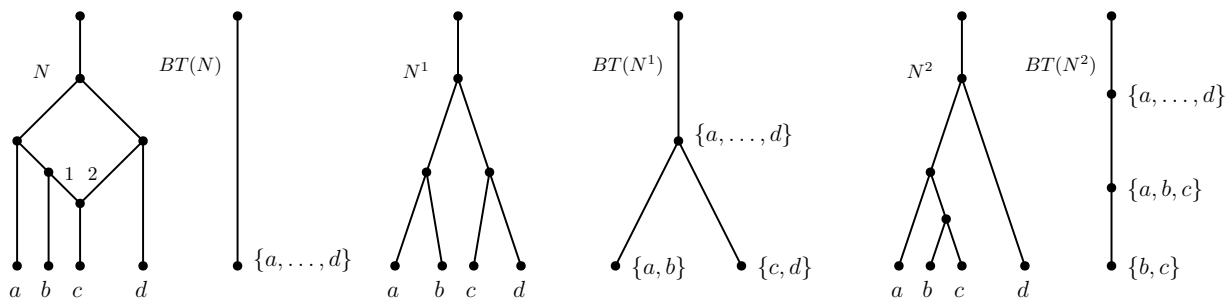}
 \caption{A tree-child network~$N$, its maximum subnetworks~$N^1, N^2$ obtained from deleting edges~1 and~2 respectively, together with their blob trees.}
 \label{fig:BlobTreeExample}
\end{figure}

We refer to the top nodes of blobs as \emph{pure nodes}.
In the case of a level-0 blob, this top node is simply the tree node itself.
Let~$x$ denote the pure node of some blob~$B$ of some network~$N$.
Then~$desc_N(x) = desc_N(B)$ denotes the set of leaf-descendants of~$x$ (and thus~$B$) in~$N$.

For a general network~$N$, it is possible for~$BT(N)$ to contain two nodes with the same label if there is a blob in~$N$ of indegree-1 and outdegree-1.
However the same cannot occur in tree-child networks, due to the following lemma.

\begin{lemma}\label{lem:TCLeafDescSet}
 Let~$N$ be a tree-child network on~$X$, let~$A\subseteq X$ and let~$x$ be a highest tree node with~$desc_N(x) = A$.
 If a tree node~$y\neq x$ also has~$desc_N(y) = A$ then one child of~$x$ is a reticulation~$r$ such that~$y$ is below~$x$ and~$y$ is above~$r$. Hence,~$x$ is the unique highest tree node with~$desc_N(x) = A$ and all other tree nodes~$y$ with ~$desc_N(y) = A$ are in the same blob.
\end{lemma}
\begin{proof} Let $y\neq x$ be a tree node with~$desc_N(y) = A$.
 To begin, note that~$y$ must be either above or below~$x$. To see this, note that by the tree-child property of~$N$, there exists a leaf~$l$ that is reached by~$x$ via a tree path.
 Then for~$y$ to be an ancestor of~$l$,~$y$ must be either above or below~$x$.
 Hence,~$x$ is the unique highest tree node with~$desc_N(x) = A$ and~$y$ is below~$x$. 
 
 By the tree-child property of~$N$, either~$x$ can have two children that are tree nodes or leaves, or~$x$ can have one tree node or leaf child and one reticulation child.
 Let~$c_1,c_2$ denote the children of~$x$, and by the tree-child property of~$N$, there exist leaves~$l_1, l_2$ that are reached by~$c_1,c_2$ via tree paths respectively.
 
 First suppose that the two children~$c_1,c_2$ of~$x$ are tree nodes or leaves.
 Then for~$y$ to be an ancestor of both~$l_1$ and~$l_2$,~$y$ must be an ancestor of both~$c_1$ and~$c_2$, contradicting that~$y\neq x$ is below~$x$.
 
 Hence, one of the two children of~$x$, is a reticulation~$r$. Without loss of generality,~$r=c_1$. It remains to show that~$y$ is above~$c_1$. Since~$y$ is an ancestor of~$l_1$, and there is a tree path from~$c_1$ to~$l_1$, node~$y$ is either above or below~$c_1$. Suppose for contradiction that~$y$ is below~$c_1$. Since~$y$ is also an ancestor of~$l_2$, there exists a directed path from~$y$ to~$l_2$. This path must pass through~$x$ since the path from~$x$ to~$l_2$ is a tree path. This is a directed path from~$y$ to~$x$. However, since there is also a directed path from~$x$ to~$y$ (via~$c_1$), and~$y\neq x$, it follows that there exists a directed cycle, a contradiction.
\end{proof}


The following corollary follows immediately from Lemma~\ref{lem:TCLeafDescSet}.

\begin{corollary}\label{cor:BTUniqueLabel}
 Let~$N$ be a tree-child network.
 Then its blob tree~$BT(N)$ contains nodes with unique labels.
\end{corollary}

Due to this, we identify blob nodes by their node labels, e.g., for a blob~$B$ in~$N$ with~$desc_N(B) = A$, the corresponding blob node in~$BT(N)$ is~$A$.

\subsection{On Reticulated Cherries}\label{subsec:RetCherries}
\review{Let~$x$ and~$y$ be two non-reticulation nodes in a network~$N$.
Let~$p_x,p_y$ be the parents of~$x,y$ respectively, where~$p_y$ is a reticulation and~$p_x$ is a parent of~$p_y$.
Let~$g_y$ denote the parent of~$p_y$ that is not~$p_x$ (see Figure~\ref{fig:RetCherryIsolateBT} (a)).
We call the subgraph of~$N$ induced by the nodes~$x,y,p_x,p_y,$ and~$g_y$ a \emph{reticulated cherry shape}.
We will refer to the reticulated cherry shape as~\review{$\langle x,y\rangle$} and say that the nodes~$x$ and~$y$ \emph{form} the reticulated cherry shape.
In this case we say that the reticulation is on~$y$ and that the reticulation~$p_y$ is in the reticulated cherry shape~\review{$\langle x,y\rangle$}.
}
This notion is a generalization of the reticulated cherries defined by Bordewich et al. in~\cite{bordewich2018recovering}, in which both~$x$ and~$y$ are leaves.


\begin{lemma}\label{lem:blobseesretcherryshape}
 In a tree-child network, all reticulations are in a \emph{reticulated cherry shape}.
 Moreover, for~$k\geq 1$, there is at least one reticulation in each level-$k$ blob that is in a reticulated cherry shape formed by two nodes outside of the blob.
\end{lemma}

\begin{proof}
 Let~$N$ be a tree-child network and consider a reticulation~$r$ in an arbitrarily chosen blob~$B$.
 By the tree-child property,~$r$ must have a non-reticulation child~$y$ and two tree node parents~$t_1,t_2$.
 The child of~$t_1$ that is not~$r$ must be a non-reticulation~$x$.
 Then~$r$ is in a reticulated cherry shape formed by~$x$ and~$y$.
 
 Now consider a lowest tree node~$a$ in~$B$.
 If both children of~$a$ were to be non-reticulations then at least one of the children would also be contained in~$B$, contradicting our choice of~$a$.
 If both children of~$a$ were to be reticulations then the network would no longer be tree-child, a contradiction.
 Thus, one child of~$a$ is a reticulation, say~$c$, and the other a non-reticulation, say~$x$.
 The child of~$c$, say~$y$, must be a non-reticulation as the network is tree-child, and thus~$B$ contains a reticulated cherry shape formed by two nodes~$x,y$. Moreover,~$x$ and~$y$ are outside of~$B$ because they are either leaves or tree nodes, and below a lowest tree node in~$B$.
\end{proof}

A reticulated cherry shape~\review{$\langle x,y\rangle$} is called a \emph{lowest reticulated cherry shape of a blob~$B$}, if the parent~$p_x$ of~$x$ is a lowest tree node of~$B$. This implies that~$x$ and~$y$ are not contained in~$B$, as shown in the proof of Lemma~\ref{lem:blobseesretcherryshape}.

Suppose we are given a reticulated cherry shape~\review{$\langle x,y\rangle$} with the reticulation on~$y$ and let~$g_y$ be the parent of~$p_y$ that is not~$p_x$.
We use the following operations defined by Bordewich et al.~\cite{bordewich2018recovering}.
\begin{itemize}
 \item \emph{cutting}~\review{$\langle x,y\rangle$} is the operation of deleting~$(p_x,p_y)$ and suppressing~$p_x$ and~$p_y$.
 \item \emph{isolating}~\review{$\langle x,y\rangle$} is the operation of deleting~$(g_y,p_y)$ and suppressing~$g_y$ and~$p_y$.
\end{itemize}

Let~$N'$ be a maximum subnetwork of a tree-child network~$N$ obtained by isolating a lowest reticulated cherry shape of a blob~$B$.
Then there is a pure node in~$N'$ that is not a pure node in~$N$ (Figure~\ref{fig:RetCherryIsolateBT}). Moreover, if blob~$B$ is of level at least~2, the leaf-descendant set of the new pure node is not equal to the leaf-descendant set of any node in~$N$.
This leads to the following observation.

\begin{observation}\label{obs:DiffBlobTreeAfterDeletion}
 For a tree-child network~$N$ and~$B$ a level-$k$ blob, with~$k\geq 2$, there is always a reticulation edge we can delete from~$B$ such that the blob tree of the resulting subnetwork is not equal to~$BT(N)$.
\end{observation}

\begin{figure}[h!]
 \centering
 \includegraphics[scale = 0.5]{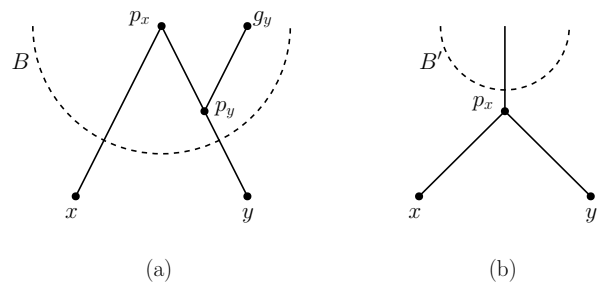}
 \caption{(a) A portion of the network showing a lowest reticulated cherry shape in a blob~$B$.
 (b) The same portion of the network after isolating the reticulate cherry shape~\review{$\langle x,y\rangle$}. Note here that~$p_x$ is a pure node in the subnetwork, but~$p_x$ is not a pure node in the original network.
 }
 \label{fig:RetCherryIsolateBT}
\end{figure}

Now suppose~$x$ and~$y$ are both leaves.
If~$x$ and~$y$ share a common parent, then they form a \emph{cherry}.
If~$x$ and~$y$ forms a reticulated cherry shape then it is a \emph{reticulated cherry}.
The following Lemma from Bordewich and Semple~\cite{bordewich2016determining} is essential for our results and will be used extensively throughout the text.

\begin{lemma}[\cite{bordewich2016determining}]\label{lem:BordewichLemma}
 If~$N$ is a tree-child network on at least two leaves, then~$N$ contains either a cherry or a reticulated cherry.
\end{lemma}

\subsection{Reconstructing the Blob Tree of a Tree-child network}\label{subsec:BlobTreeReconstruction}

\begin{lemma}\label{lem:BT(N)ContainsAThenBT(N')AlsoDoes}
 For a valid network~$N$, if the blob tree~$BT(N)$ contains a blob node~$A$ then~$BT(N')$ contains the blob node~$A$ for \review{every} maximum subnetwork~$N'$ of~$N$.
\end{lemma}
\begin{proof}
 First suppose that~$A$ is a blob node corresponding to a level-0 blob in~$N$.
 The corresponding node~$t$ in~$N$ is not incident to any reticulation edges, so it is not possible to suppress~$t$ via a reticulation edge deletion.
 Note that a reticulation edge deletion from a blob above or below~$t$ would not change the leaf-descendant set of~$t$.
 Hence~$t$ remains a level-0 blob in all maximum subnetworks of~$N$ with leaf-descendant set~$A$.
 Thus~$A$ is a blob node in all~$BT(N')$ for all maximum subnetworks~$N'$ of~$N$.

 Now suppose that~$A$ is a blob node corresponding to a blob of level at least~$1$.
 Suppose~$t$ is the corresponding pure node in~$N$.
 If~$t$ is not incident to a reticulation edge then there is no way of suppressing~$t$ by means of edge deletions and any reticulation edge deletion will not change the leaf-descendant set of~$t$. Hence,~$t$ is a pure node with leaf-descendant set~$A$ in all maximum subnetworks of~$N$.
 If on the other hand there is a reticulation~$r$ with edges~$(t,r),(s,r)$ then let~$c$ be the child of~$t$ that is a tree node (it is possible that~$c=s$).
 Because~$t$ is the top node of the blob, there is a directed path from~$t$ to~$s$, which must include~$c$. Hence there is a directed path from~$c$ to~$s$ and to~$r$. Therefore, we have~$desc_N(r)\subseteq desc_N(c)$.
 So~$desc_N(c) = desc_N(c)\cup desc_N(r) = desc_N(t) = A$.
 We now use the fact that, after a valid edge deletion, only the endpoints of the edge are suppressed in the resultant maximum subnetwork.
 The maximum subnetwork where~$(t,r)$ is deleted contains~$c$ as a pure node, and hence~$A$ is a blob node in its blob tree.
 The maximum subnetwork where~$(s,r)$ is deleted contains~$t$ as a pure node, and hence~$A$ is a blob node in its blob tree.
 The maximum subnetwork where some other reticulation edge is deleted contains~$t$ as a pure node, and hence~$A$ is a blob node in its blob tree.
 Thus~$A$ is a blob node in~$BT(N')$ for all maximum subnetworks~$N'$ of~$N$.\\
\end{proof}

\begin{lemma}\label{lem:BT(N')ContainsAThenBT(N)ContainsA}
 For a valid network~$N$, if~$BT(N')$ contains a blob node~$A$ for \review{every} maximum subnetwork~$N'$ of~$N$ then~$BT(N)$ also contains the blob node~$A$.
\end{lemma}
\begin{proof} 
 Consider some lowest reticulation~$r$ in~$N$ such that~$r$ is the ancestor of some~$a \in A$. 
 Let~$c$ be the child of~$r$ in~$N$.
 Since~$r$ is of outdegree-1, we have~$desc_N(r) = desc_N(c)$.
 We may assume~$desc_N(r)\neq A$, as otherwise~$c$ is the root of a pendant subtree spanning~$A$ in~$N$, and consequently~$A$ is a blob node in~$BT(N)$.
 Let~$(u,r),(v,r)$ be the edges leading into~$r$.
 Let~$N', N''$ be the maximum subnetworks of~$N$ obtained by deleting~$(u,r), (v,r)$ respectively.
   Note here that every node~$x$ in~$N'$ or~$N''$ is also a node in~$N$.
 We now examine the relations between~$desc_N(r)$ and~$A$ exhaustively.
 \begin{itemize}
  \item Suppose~$desc_N(r)\not\subset A$ and~$A\not\subset desc_N(r)$.
   We show that there is no node in~$N'$ that has leaf-descendant set~$A$.
   By assumption, there exists a node~$a'\in desc_N(r)$ such that~$a'\notin A$. Then,~$a'\in desc_{N'}(c)$.
   Let~$x$ be a node in~$N'$ (which is also a node in~$N$).
   We examine the relations between~$x$ and~$c$ in~$N'$ exhaustively.
   \begin{itemize}
    \item If~$x$ is an ancestor of~$c$ in~$N'$ then~$desc_{N'}(x)\neq A$ since~$a'\in desc_{N'}(c)\subseteq desc_{N'}(x)$.
    \item If~$x$ is a descendant of~$c$ in~$N'$ then~$desc_{N'}(x)\neq A$ since~$A\not\subset desc_{N'}(c)$.
    \item If~$x$ is incomparable to~$c$ in~$N'$ then~$desc_{N'}(x)\neq A$ since~$a\notin desc_{N'}(x)$ by assumption that~$r$ was the lowest reticulation above~$a$.
   \end{itemize}
   It follows that~$A$ is not in~$BT(N')$, and this case is not possible.
   The only possibilities then, are either~$A\subsetneq desc_N(r)$ or~$desc_N(r)\subsetneq A$.
 \end{itemize}

 By assumption,~$BT(N')$ and~$BT(N'')$ both contain~$A$.
 Because of this, there are corresponding pure nodes~$x', x''$ in~$N', N''$ (also in~$N$) respectively with~$desc_{N'}(x') = desc_{N''}(x'') =A$.
 
 \begin{itemize}
  \item Suppose~$A\subsetneq desc_N(r)$.
  Then~$x'$ must be a descendant of~$c$ in~$N'$, implying that~$x'$ must be a descendant of~$r$ in~$N$. 
  We claim that~$x'$ is a pure node in~$N$ with~$desc_N(x')=A$.
  If~$x'$ is not a pure node in~$N$ then there exists a reticulation~$s\neq r$ below~$x'$ where~$s$ and~$x'$ are contained in the same blob in which~$x'$ is not the top-node, in~$N$.
  The edge deletion does not suppress or delete the node~$s$, since~$s$ is a descendant of~$r$, and any directed path from~$r$ to~$s$ is of length at least~2.
  Then,~$s$ is a reticulation that is below~$r$ such that the leaf-descendant set of~$s$ contains an element of~$A$.
  This contradicts our choice of~$r$, so~$x'$ must be a pure node in~$N$.
  Furthermore, we must have~$desc_N(x') = desc_{N'}(x') = A$ where the first equality holds as deleting a reticulation edge from above a node does not change its leaf-descendant set in the resultant subnetwork.
  Then~$x'$ must be a pure node in~$N$ with~$desc_N(x')=A$ and we are done.
  \item So we may assume~$desc_N(r)\subsetneq A$.
 We now claim that~$desc_N(v)\subseteq A$.
 Suppose not.
 Noting that~$v$ is not suppressed in~$N'$ (since~$N$ is a valid network), and since~$desc_{N'}(v) = desc_N(v)$, we split into the three possible cases for the relation between~$x'$ and~$v$ in~$N'$.
 \begin{itemize}
  \item If~$x'$ is an ancestor of~$v$ then it is also an ancestor of~$b\notin A$ in~$N'$ for some~$b\in desc_{N'}(v)$, a contradiction.
  \item If~$x'$ is incomparable to~$v$ then~$x'$ is also incomparable to~$c$ in~$N'$.
  Then, since~$a\in A$ is a leaf-descendant of~$x'$ in~$N'$, there is a reticulation~$s$ below~$r$ in~$N$ such that~$s$ is an ancestor of~$a$, which contradicts our choice of~$r$.
  \item If~$x'$ is a descendant of~$v$ then it must either be incomparable to or be a descendant of~$c$ in~$N'$.
  \begin{itemize}
   \item If~$x'$ is incomparable to~$c$ in~$N'$, then we reach a contradiction by the same argument as above.
   \item If~$x'$ is a descendant of~$c$ in~$N'$, then as~$desc_{N'}(c)\subsetneq A$ (since~$desc_N(r) = desc_{N'}(c)$) we have that~$desc_{N'}(x')\subsetneq A$, a contradiction.
  \end{itemize}
 \end{itemize}
 Thus we have that~$desc_N(v)\subseteq A$.
 By an analogous reasoning on~$x''$ in~$N''$, we have that~$desc_N(u)\subseteq A$.
 It follows that~$x'$ must be an ancestor of~$v$ in~$N'$, and so~$x'$ must be an ancestor of~$v$ in~$N$.
 It also follows that~$x'$ must be an ancestor of~$u$ in~$N$ to ensure that there is a path from~$x'$ to the leaf-descendants of~$u$ in~$N'$.
 
 We now claim that~$x'$ is also a pure node in~$N$ with leaf-descendant set~$A$.
 Indeed, adding the edge~$(u,r)$  to~$N'$ (after undoing any cleaning up) only joins descendants of~$x'$, implying~$x'$ has leaf-descendant set~$A$ in~$N$. 
 Furthermore, it cannot add any nodes that are not descended from~$x'$ to the blob containing~$x'$.
 It follows that~$x'$ remains a pure node in~$N$ with leaf-descendant set~$A$.
 \end{itemize}
\end{proof}

By combining the previous two lemmas, we see that the blob trees of valid networks are reconstructible from their maximum subnetworks.

\begin{theorem}\label{thm:BTNodeLabelPreservation}
 For a valid network~$N$, given a set~$A \subseteq X$, the blob tree~$BT(N)$ contains \review{the} blob node~$A$ if and only if~$BT(N')$ contains \review{the} blob node~$A$ for \review{every} maximum subnetwork~$N'$ of~$N$.
\end{theorem}
 
\begin{proof}
 Follows from Lemmas~\ref{lem:BT(N)ContainsAThenBT(N')AlsoDoes} and~\ref{lem:BT(N')ContainsAThenBT(N)ContainsA}.
\end{proof}

We can prove a similar result for MLLSs.

\begin{theorem}\label{thm:BTNodeLabelPreservationMLLS}
 Let~$N$ be a level-$k$ valid network, with~$k\geq 1$.
 Given a set~$A \subseteq X$, the blob tree~$BT(N)$ contains the blob node~$A$ if and only if~$BT(N^{mlls})$ contains~$A$ for \review{every} MLLS~$N^{mlls}$ of~$N$.
\end{theorem}
\begin{proof}
 Suppose first that the blob tree~$BT(N)$ contains the node~$A$, and let~$N^{mlls}$ be an MLLS of~$N$ obtained by deleting the edges in the set~$E = \{e_1,\dots,e_m\}$.
 Consider the maximum subnetwork~$N'$ of~$N$ obtained by deleting the reticulation edge~$e_1$.
 By Theorem~\ref{thm:BTNodeLabelPreservation},~$BT(N')$ contains the blob node~$A$.
 Now consider the maximum subnetwork~$N''$ of~$N'$ obtained by deleting the reticulation edge~$e_2$.
 Then~$BT(N'')$ contains the blob node~$A$ by Theorem~\ref{thm:BTNodeLabelPreservation}.
 Continuing in this fashion for all edges in~$E$ shows that~$BT(N^{mlls})$ contains the blob node~$A$.\\
 
 Now suppose that~$A$ is not a blob node of~$BT(N)$.
    We prove that then there exists an MLLS~$N^{mlls}$ of~$N$ such that~$BT(N^{mlls})$ does not contain the blob node~$A$.
    Let~$B$ denote the blob in~$N$ with leaf-descendant set~$D$, such that~$D$ is the smallest set that contains~$A$.
    Consequently, if there exists a pure node in an MLLS~$N^{mlls}$ of~$N$ with leaf-descendant set~$A$, then it must be a node that was originally in the blob~$B$.
    Now observe that deleting reticulation edges from blobs that are not~$B$ do not affect the leaf-descendant set of nodes in~$B$.
    Then we may assume, without loss of generality, that~$N$ is a single blob network.
    But then by Theorem~\ref{thm:BTNodeLabelPreservation}, $A$ is not a blob node in~$BT(N^{mlls})$, for some MLLS~$N^{mlls}$ of~$N$.
\end{proof}

We call a set~$A\subseteq X$ a \emph{foundation node} of~$N$ if~$BT(N)$ contains the node~$A$.
Let~$\cF(N)$ be the set of all foundation nodes of~$N$.

\begin{theorem}\label{thm:BloBTreeReconstruction}
 For a level-$k$ valid network~$N$, with~$k\geq 1$, its blob tree~$BT(N)$ is reconstructible from its MLLSs.
\end{theorem}

\begin{proof}
 By Theorem~\ref{thm:BTNodeLabelPreservationMLLS}, the set of all foundation nodes~$\cF(N)$ consists of the blob nodes that appear in \review{$BT(N^{mlls})$ for every MLLS~$N^{mlls}$ of~$N$.} 
 Then, the blob tree~$BT(N)$ is the tree with vertex set~$\cF(N)$ and an edge~$(A,B)$ precisely if~$B\subsetneq A$ and there is no~$C\in\cF(N)$ with~$B\subsetneq C\subsetneq A$.
\end{proof}

\subsection{Minimum number of MLLSs to reconstruct the Blob Tree of a Tree-child network}\label{subsec:MinNumMLLSBT}

 \setcounter{claimcounter}{0}
 \setcounter{claimproofcounter}{0}

We consider the minimum number of MLLSs required to reconstruct the blob tree of a tree-child network.
Let~$r$ be some reticulation in a blob~$B$.
We call a node~$s$ a \emph{pseudo pure node of~$r$} if it is \review{a} lowest node in~$B$ such that there are two edge disjoint directed paths from~$s$ to~$r$.

\begin{lemma}\label{lem:MinNumberMStoReconstructBT}
 Let~$N$ be a level-$k$ tree-child network where~$k\geq1$.
 Two maximum subnetworks~$N'$ and~$N''$ of~$N$ suffice to reconstruct~$BT(N)$.
\end{lemma}
\begin{proof}
 Let~$r$ be a lowest reticulation in some blob~$B$.
 Let~\review{$\langle x,y\rangle$} denote the reticulated cherry shape \review{that contains}~$r$. Let~$p_x$ and~$p_y=r$ be the parents of~$x$ and~$y$, respectively, and let~$g_y$ be the parent of~$p_y$ that is not~$p_x$.
 Let~$N'$ and~$N''$ be the maximum subnetworks of~$N$ derived by cutting and isolating~$\langle x,y\rangle$ respectively.
 \review{Let~$F'$ and~$F''$ denote the set of foundation nodes of~$N'$ and~$N''$, respectively, that are not foundation nodes of~$N$.}
 We claim that the intersection of \review{~$F'$ and~$F''$} is empty, from which it follows that the intersection of the node sets of~$BT(N')$ and~$BT(N'')$ contains the foundation nodes of~$N$. Since by Lemma~\ref{lem:BT(N)ContainsAThenBT(N')AlsoDoes} each foundation node of~$N$ is a foundation node of each maximum subnetwork, it follows that the intersection of the node sets of~$BT(N')$ and~$BT(N'')$ is precisely the set of \review{all} foundation nodes of~$N$. 
 
 \review{Let~$P'$ denote the set of all pure nodes in~$N'$ that have leaf-descendant sets in~$F'$.
 Similarly let~$P''$ denote the set of all pure nodes in~$N''$ that have leaf-descendant sets in~$F''$.
 We prove the following claims regarding the pure nodes of~$P'$ and~$P''$.}\\
 
 \begin{claime}
  Let~$p\in P'$ ($p\in P''$). 
  Then~$p$ is an ancestor of~$r$ in~$N$.
 \end{claime}
 
 \begin{claimproof}
  Suppose not. 
  First suppose that~$p$ is a descendant of~$r$ in~$N$.
  As~$r$ is a lowest reticulation in~$N$,~$p$ is a tree node or a leaf in~$N'$ ($N''$).
  If~$p$ is a tree node then~$p$ must have been a pure node in~$N$ to begin with: the pendant subnetwork rooted at the child of~$r$ is an invariant upon obtaining maximum subnetworks of~$N$, since~$r$ is a lowest reticulation.
  This contradicts the fact that~$p$ is an element of~$P'$ ($P''$).
  If~$p$ is a leaf then~$p$ cannot be a pure node in~$N'$ ($N''$), a contradiction. 

  Now suppose that~$p$ is incomparable to~$r$ in~$N$.
  Let~$p$ be the pure node of a blob~$B'$ in~$N'$ ($B''$ in~$N''$).
  As~$p$ is not an ancestor of~$r$ in~$N$,~$p$ must also not be an ancestor of~$p_x$ nor~$g_y$ in~$N$.
  We see that~$B'$ ($B''$) remains a blob after adding the edge~$(p_x,r)$ to~$N'$ ($(g_y,r)$ to~$N''$), and so~$p$ remains a pure node in~$N$, a contradiction.
  
 \end{claimproof}
 
 
 \begin{claime}
  Let~$p\in P'$ ($p \in P''$), \review{and let~$s$ be a pseudo pure node of~$r$ in~$N$.}
  Then~$p$ is a descendant of~$s$ in~$N$ and~$p\neq s$.
 \end{claime}
 
 \begin{claimproof}
  Suppose not.

  If~$p$ is equal to or strictly above~$s$ then~$p$ is an ancestor of both~$p_x$ and~$g_y$.
  Adding the edge~$(p_x,r)$ to~$N'$ ($(g_y,r)$ to~$N''$) only joins descendants of~$p$. 
  Furthermore, it cannot add any nodes that are not descended from~$p$ to the blob in~$N'$ ($N''$) containing~$p$.
  It follows that~$p$ remains a pure node in~$N$, a contradiction.

  Now suppose that~$p$ is incomparable to~$s$.
  If~$p$ is not in the blob~$B$, then as reticulation edge deletions do not affect other blobs, we have that~$p$ must have been a pure node in~$N$.
  This contradicts our assumption on~$p$.
  So~$p$ must be in the blob~$B$.
  Since~$p$ is incomparable to~$s$, but~$p$ must still be an ancestor of~$r$ by Claim~1,~$p$ must be an ancestor of a reticulation~$r'$ such that~$s$ is an ancestor of~$r'$ and~$r'$ is an ancestor of~$r$.
  The parent of~$s$, denoted~$p_s$, is not suppressed in both~$N'$ and~$N''$. 
  Now~$p_s$ is either above or incomparable to~$p$, and the two nodes belong to the blob which contains~$r'$ in~$N'$ ($N''$) \review{(see Figure~\ref{fig:PureNodeDescendantOfPseudo} (a))}.
  It follows that~$p$ cannot be a pure node in~$N'$ ($N''$), which contradicts our assumption.
  
 \end{claimproof}
 
 It remains to show that \review{~$F' \cap F'' = \emptyset$}. 
 Let~$P'_x = \{p \in P' : p \text{ is an ancestor of } p_x 
 \text{ in } N\}$ and let~$P'_y = \{p \in P' : p \text{ is an ancestor of } g_y 
 \text{ in } N\}$.
 By Claim~1, we have~$P' = P'_x \cup P'_y$.
 By Claim~2, we have~$P'_x \cap P'_y = \emptyset$.
 \review{Let~$P''_x = \{p \in P'' : p \text{ is an ancestor of } p_x 
 \text{ in } N\}$ and let~$P''_y = \{p \in P'' : p \text{ is an ancestor of } g_y 
 \text{ in } N\}$.
 Similarly we have~$P'' = P''_x \cup P''_y$ and~$P''_x \cap P''_y = \emptyset$.}
 Let~$a'\in P'_x, a''\in P''_x$, and~$b''\in P''_y$.
 Let~$u\in desc_N(x)$ and let~$v\in desc_N(y)$.
 Clearly,~$u\in desc_{N'}(a')$ and~$u,v\in desc_{N''}(a'')$.
 By Claim~2,~$v\notin desc_{N'}(a')$ and~$u,v\notin desc_{N''}(b'')$ \review{(see Figure~\ref{fig:PureNodeDescendantOfPseudo} (b))}. 
 \review{This implies that for some~$A\in F'$ such that~$u\in A$, we have~$A\notin F''$.
 An analogous argument shows that for some~$B\in F'$ such that~$v\in B$, we have~$B\notin F''$.
 Because of the way in which we defined the network~$N'$, all foundation nodes in~$F'$ must contain the element~$u$ or~$v$, but not both.
 Thus the above two cases cover all foundation nodes in~$F'$; therefore,~$F'$ and~$F''$ are disjoint.}
\end{proof}

\begin{figure}[h]
 \centering
 \includegraphics[scale = 0.6]{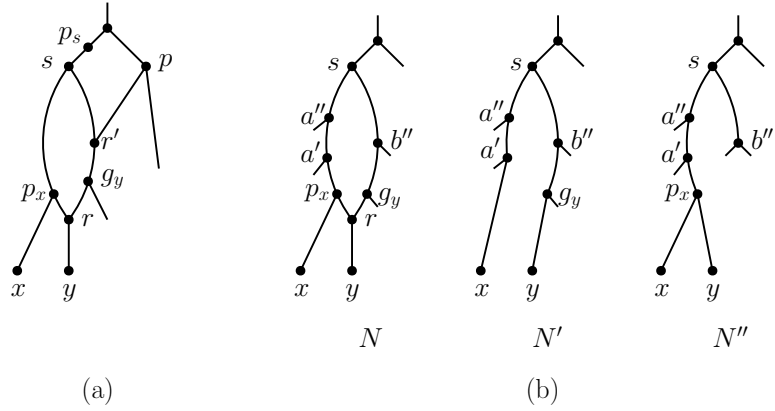}
 \caption{Visual aid for Lemma~\ref{lem:MinNumberMStoReconstructBT} proof.
 (a) Proof of Claim~2. 
 Cutting or isolating~\review{$\langle x,y\rangle$} results in subnetworks where~$p_s$ and~$p$ lie in the blob containing~$r'$.
 (b) Proof of the paragraph after Claim~2, which shows that~$P'\cap P'' = \emptyset$.
 }
 \label{fig:PureNodeDescendantOfPseudo}
\end{figure}

Let~$N$ be a network and let $N_A$ be a pendant subnetwork of~$N$ rooted by a node with leaf-descendant set~$A$.
\emph{Collapsing}~$N_A$ from~$N$ means that we replace~$N_A$ by a leaf~$A$. 
Let~$N\setminus N_A$ denote the network obtained by collapsing~$N_A$ from~$N$.

\begin{lemma}\label{lem:BTSetminusPendantSubtree}
 Let~$N$ be a tree-child network, and let~$N_A$ denote a pendant subnetwork of~$N$ rooted at a node with leaf-descendant set~$A$.
 Then $BT(N\setminus N_A)$ is obtained from $BT(N)$ by deleting the pendant subtree rooted at~$A$.
\end{lemma}
\begin{proof}
 By definition, there exists a blob node~$A$ in~$BT(N)$.
 Note that pendant subnetworks of~$N$ uniquely correspond to a pendant subtree of~$BT(N)$, by definition of blob trees and also because node labels of blob trees are unique for tree-child networks (Corollary~\ref{cor:BTUniqueLabel}).
 Then the pendant subtree of~$BT(N)$ rooted at~$A$ is uniquely defined by~$N_A$ and vice versa: this implies the lemma.
\end{proof}

\begin{lemma}\label{lem:MinNumberMLLStoReconstructBT}
 Let~$N$ be a level-$k$ tree-child network with~$k\geq 2$.
 Two MLLSs~$N^{mlls}_1$ and~$N^{mlls}_2$ of~$N$ suffice to reconstruct~$BT(N)$.
 In particular,~$N^{mlls}_1$ is the MLLS obtained by cutting a lowest reticulated cherry shape in every level-$k$ blob, and~$N^{mlls}_2$ is the MLLS obtained by isolating these reticulated cherry shapes.
\end{lemma}

 \begin{proof}
 We prove the lemma by induction on the number of level-$k$ blobs~$l$ in~$N$.
 For the base case, there is only one level-$k$ blob in~$N$. 
 By Lemma~\ref{lem:MinNumberMStoReconstructBT}, we are done.
 
 So suppose now that~$N$ contains~$l \geq 2$ level-$k$ blobs.
 Consider a lowest level-$k$ blob~$B$ in~$N$, and let~$A$ denote the leaf-descendant set of~$B$.
 Let~$N_A$ and~$N^{mlls}_{iA}$ denote the pendant subnetwork of~$N$ and~$N^{mlls}_i$ rooted at the pure node with leaf-descendant set~$A$, for~$i=1,2$.
 By Theorem~\ref{thm:BTNodeLabelPreservationMLLS},~$A$ is a blob node in~$BT(N^{mlls}_i)$ for~$i=1,2$, and therefore such pendant subnetworks exist.
 Note that the pendant subnetworks~$N^{mlls}_{iA}$ are maximum subnetworks of~$N_A$ obtained by cutting and isolating the reticulated cherry shape associated with some lowest reticulation~$r$.
 By Lemma~\ref{lem:MinNumberMStoReconstructBT}, we have that~$N^{mlls}_{1A}$ and~$N^{mlls}_{2A}$ suffice to reconstruct~$BT(N_A)$.
 We now collapse~$N^{mlls}_{iA}$ from the MLLS~$N^{mlls}_i$ for~$i=1,2$.
 Furthermore we collapse~$N_A$ from the network~$N$.
 Note that~$N\setminus N_A$ is a level-$k$ tree-child network with~$l-1$ level-$k$ blobs, and that~$N^{mlls}_i\setminus N^{mlls}_{iA}$ are MLLSs of~$N\setminus N_A$ obtained by cutting and isolating a lowest reticulated cherry shape from every level-$k$ blob, for~$i=1,2$ respectively.
 By the induction hypothesis, these two MLLSs of~$N\setminus N_A$ suffice to reconstruct~$BT(N\setminus N_A)$.
 Now by Lemma~\ref{lem:BTSetminusPendantSubtree}, we have that~$BT(N)$ is the blob tree obtained by appending~$BT(N_A)$ to~$BT(N\setminus N_A)$.
 We append~$BT(N_A)$ to the node~$C$ in~$BT(N\setminus N_A)$, such that~$A\subseteq C$, and there exists no node~$D\in BT(N\setminus N_A)$ where~$A\subseteq D\subseteq C$.
\end{proof}

 \review{Given~$N, N^{mlls}_1$, and~$N^{mlls}_2$ as in the setting of Lemma~\ref{lem:MinNumberMLLStoReconstructBT}, the foundation nodes of~$N$ can be found by taking the intersection of the foundation nodes of~$N^{mlls}_1$ and that of~$N^{mlls}_2$.
 Then,~$BT(N)$ can be reconstructed as in the proof of Theorem~\ref{thm:BloBTreeReconstruction}.}

\subsection{Identifying the level-$k$ blobs of a Tree-child network}\label{subsec:BlobIdentification}

We now show that given the MLLSs, it is possible to identify which foundation nodes correspond to a level-$k$ blob in the original tree-child network.

\begin{lemma}\label{lem:blobgraphnew12}
Let~$N$ be a level-$k$ tree child network with~$k\geq 2$. A blob of~$N$ is level-$k'<k$ if and only if the  set of \review{children} of the corresponding blob node in~$BT(N^{mlls})$, for every~$N^{mlls}\in \cN^{mlls}(N)$, is precisely the set of \review{children} of the blob node in~$BT(N)$.
\end{lemma}

\begin{proof}
 Suppose~$B$ is a level-$k'<k$ blob in~$N$.
 Then~$B$ remains intact (no reticulation edges deleted) in all MLLSs of~$N$.
 Let~$B'$ be a blob in~$N$ that is directly below~$B$, and let~$e$ denote the outgoing edge from~$B$ to the pure node of~$B'$.
 The edge~$e$ is not suppressed in any MLLS of~$N$.
 And since edges are deleted to obtain MLLSs of~$N$, we have that the number of leaves that are below the edge~$e$ (below the child of~$e$) stays the same since blobs are biconnected.
 By Theorem~\ref{thm:BTNodeLabelPreservationMLLS}, every node in~$BT(N)$ is a node in~$BT(N^{mlls})$ for all MLLSs~$N^{mlls}$ of~$N$.
 Furthermore for tree-child networks, the node labels in blob trees are unique.
 Then the blob node of~$B$ must have the blob node of~$B'$ as one of its \review{children} in the blob tree of all MLLSs.
 Since~$B'$ was chosen arbitrarily, this implies that the \review{set of children} of the blob node in~$BT(N^{mlls})$, for every~$N^{mlls}\in \cN^{mlls}(N)$, is precisely the \review{set of children} of the blob node in~$BT(N)$.
 
 For the other direction, we prove the contrapositive.
 Suppose~$B$ is a level-$k$ blob in~$N$, and let~$desc_N(B) = A$.
 By Observation~\ref{obs:DiffBlobTreeAfterDeletion}, we can isolate a lowest reticulated cherry in~$B$ to obtain an MLLS~$N^{mlls}$ of~$N$ where~$BT(N^{mlls})$ is different from~$BT(N)$.
 In this construction of~$N^{mlls}$, there exists a pure node in~$N^{mlls}$ which was not a pure node in~$N$.
 Then the \review{set of children} of~$A$ in~$BT(N^{mlls})$ is not the same as the \review{set of children} of~$A$ in~$BT(N)$.
\end{proof}

\begin{figure}[h!]
 \centering
 \includegraphics[width=\textwidth]{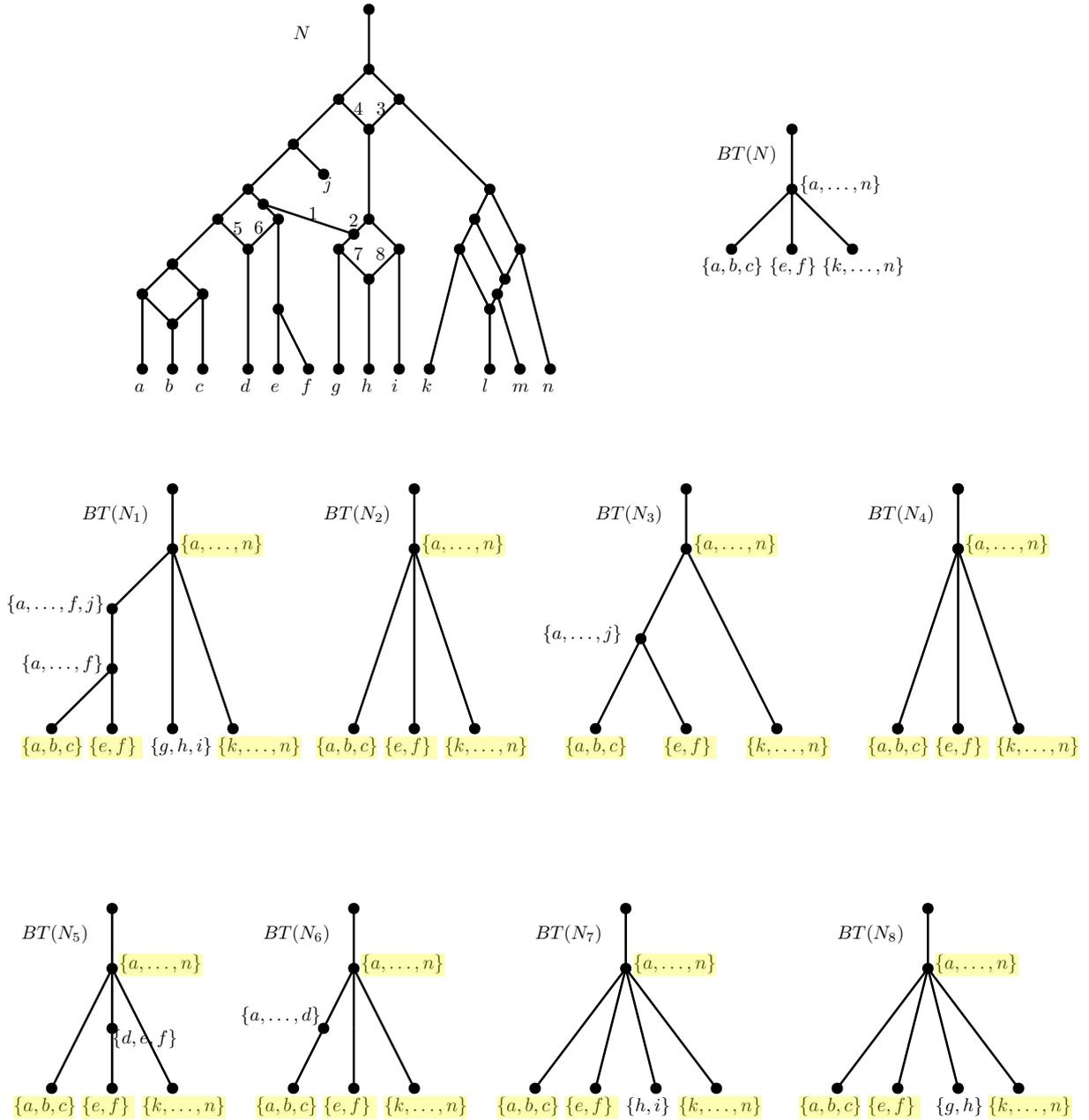}
 \caption{$N_i$ for~$i = 1,2,\dots, 8$ refers to the MLLS of a level-4 tree-child network~$N$ obtained by deleting the reticulation edge~$i$.
~$BT(N)$ is the blob tree of the network~$N$ and~$BT(N_i)$ is the blob tree of~$N_i$ for each~$i$.
 The foundation nodes are highlighted in yellow.}
 \label{fig:BTNeighborChange}
\end{figure}

Figure~\ref{fig:BTNeighborChange} illustrates Lemma~\ref{lem:blobgraphnew12} with a level-4 tree-child network~$N$.
The blob trees of its MLLSs are taken, from which the blob tree of~$N$ can be reconstructed (Theorem~\ref{thm:BTNodeLabelPreservationMLLS}). Then, it can be seen that the set of \review{children} of the blob node~$\{a,\dots,n\}$ in~$BT(N_i)$ for~$i = 1,3,5,6,7,8$ differs from the the set of \review{children} of~$\{a,\dots,n\}$ in~$BT(N)$. Hence, the blob with leaf-descendant set $\{a,\dots,n\}$ is of level-4.
Since the \review{children} of the other blob nodes do not change, the blobs with leaf-descendant sets~$\{a,b,c\}$, $\{e,f\}$ and $\{k,\dots,n\}$ are blobs of level lower than~4.

\section{Leaf pair analysis}\label{sec:LeafPairAnalysis}

In order to reconstruct a tree-child network from its MLLSs, we require a way of locating the position of the missing reticulation edges.
In this section, we show that studying the topology of a leaf pair in the MLLSs gives enough information to infer the topology of those same leaves in the original network.
The next section will show how we can use this to find the location of the missing reticulation edge of each blob by choosing the appropriate leaf pair. 

We use the inter-node distance as defined by Bordewich and Semple~\cite{bordewich2016determining}.
For our purposes, we slightly tweak the definition by allowing the endpoints to be non-leaf nodes.

\begin{definition}
 Let~$N$ be a network and let~$x,y\in N$.
 An \emph{up-down path} of length~$p$ from~$x$ to~$y$ is a sequence of nodes~$x=v_0, v_1, v_2, \dots, v_{p-1}, v_p=y$ in~$N$, such that for some~$0\leq i\leq p$,~$N$ contains the edges
 \begin{equation*}
  (v_i,v_{i-1}),\dots, (v_1,x)
 \end{equation*}
 and
 \begin{equation*}
  (v_i,v_{i+1}), (v_{i+1},v_{i+2}), \dots, (v_{p-1},y).
 \end{equation*}
 The node~$v_i$ is the \emph{apex} of this up-down path. 
 The length of a shortest~$xy$ up-down path~$P$ in~$N$ is denoted~$d_N(x,y)$.
\end{definition}

Note that the shortest up-down distance $d_N(x,y)$ in a network~$N$ may not necessarily be the shortest distance in the underlying undirected graph of~$N$ (where the underlying undirected graph of~$N$ is obtained by replacing every directed edge by an undirected edge), see Figure~\ref{fig:ShortestUDPath}.

\begin{figure}[h!]
 \centering
 \includegraphics[scale=0.5]{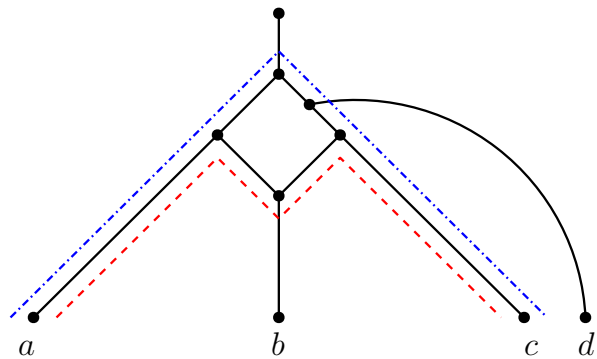}
 \caption{A network~$N$ on 4 leaves.
 The shortest~$ac$ up-down distance is 5 (blue dash-dotted path) however, the shortest~$ac$ distance in the underlying undirected graph of~$N$ is 4 (red dashed path).}
 \label{fig:ShortestUDPath}
\end{figure}

Let~$Q$ be an up-down path between nodes~$u$ and~$v$ of length at least~2 in a tree-child network~$N$.
An edge~$(u,v)$, if it exists, is called a \emph{shortcut}.  
In some papers, the notion of a shortcut (also known as a redundant arc) is defined on directed paths rather than on up-down paths~\cite{bordewich2018recovering, willson2010properties}.
For the purposes of this paper and since a directed path is by definition an up-down path (without the `up' portion), we define shortcuts on the up-down paths.
Call an up-down path which has no shortcuts in~$N$ a \emph{shortcut free} up-down path.
Note that shortest up-down paths are necessarily shortcut free.
Let~$N'$ be a maximum subnetwork of~$N$ obtained by deleting some reticulation edge~$(u,r)$. 
Let~$P'$ be an~$xy$ up-down path in~$N'$ for nodes~$x,y$.
Reinsert the edge~$(u,r)$ in~$N'$. 
Then the~$xy$ up-down path~$P'$, together with any nodes in~$\{u,r\}$ that \review{intersect} some edge of~$P'$, is called the \emph{embedded path} of~$P'$ in~$N$.

\begin{lemma}\label{lem:EdgeDeletionReducesOnly1Distance}
 In a tree-child network, deleting a single reticulation edge can reduce the up-down distance between any two leaves by at most one.
\end{lemma}
\begin{proof}
 Let~$N$ be a tree-child network and let~$N'$ be a maximum subnetwork of~$N$ obtained by deleting some reticulation edge~$(u,r)$.
 Let~$v$ be the parent of~$r$ in~$N$ that is not~$u$.
 Take any~$xy$ up-down path~$P'$ in~$N'$, and let~$P$ be its embedded path in~$N$.
 Let~$P^*$ be an up-down path in~$N$ derived from~$P$ by taking the shortcut~$(u,r)$ if it is a shortcut in~$P$.
 We show that~$|P^*|\leq |P'| + 1$.
 Now compared to~$P'$, the up-down path~$P$ contains at most 2 additional nodes - the nodes~$u$ and~$r$.
 If it contains:
 \begin{itemize}
  \item 0 additional nodes then~$(u,r)$ cannot be a shortcut of the embedded path~$P$. So,~$|P| = |P^*| = |P'|$;
  \item 1 additional node then again,~$(u,r)$ cannot be a shortcut of the embedded path~$P$. So,~$|P| = |P^*| = |P'| + 1$; 
  \item 2 additional nodes then~$(u,r)$ must be a shortcut in~$P$, as otherwise deleting~$(u,r)$ disconnects~$P'$.
  This implies that currently,~$P$ contains all three of the points~$\{u,v,r\}$.
  Then~$|P'| = |P| - 2 \geq |P^*| -1$, where the inequality follows as taking a shortcut reduces the length of an up-down path by at least~$1$.
 \end{itemize}
 It then follows that a single reticulation edge deletion from~$N$ can reduce~$d_N(x,y)$ for any two leaves~$x,y\in N$ by at most~$1$.
\end{proof}

 \setcounter{claimcounter}{0}
 \setcounter{claimproofcounter}{0}

\begin{lemma}\label{lem:PossibleShapes}
 Let~$N$ be a tree-child network. \leoR{For each pair of leaves~$\{x,y\}$, exactly one of the following cases holds (see Figure~\ref{fig:Treechildxyshapes}):}
 \begin{itemize} 
  \item \leoR{$N$ contains} a \emph{cherry}~$\Lambda(x,y)$ with nodes~$a,x,y$ and edges~$(a,x),(a,y)$;
  \item \leoR{$N$ contains} a \emph{cherry subdivided by one tree node}.
  \review{If this tree node is the parent of~$y$, there is a subgraph with nodes~$a,b,c,x,y$ and edges~$(a,x),(a,b),(b,y),(b,c)$, which we call~$\lambda(x,y)$}; 
  \item \leoR{$N$ contains} a \emph{reticulated cherry}, \leoR{which} is a cherry subdivided by one reticulation. \review{If this reticulation is the parent of~$y$, there is a subgraph with nodes~$a,b,c,i,x,y$ and edges~$(a,x),(a,c),(b,c),(c,y),(i,a)$ which we call~$K(x,y)$}; 
  \begin{itemize}
   \item if~$i=b$ then \leoR{we also} call this shape~$A(x,y)$;
   \item if~$i\neq b$ then \leoR{we also} call this shape~$H(x,y)$.
  \end{itemize}
  \item if~$d_N(x,y)\geq 4$, we say that~$N$ contains~$\Pi(x,y)$.
   \end{itemize}
  \leoR{Hence, there are eight possibilities in total:
  $\Lambda(x,y),\lambda(x,y),\lambda(y,x),A(x,y),A(y,x),H(x,y),H(y,x),\Pi(x,y)$, each of which we call a \emph{shape}. However, keep in mind that when~$N$ contains~$\Pi(x,y)$, this does not mean just that there exists an~$xy$ up-down path of length at least~4, but also that there does not exist an~$xy$ up-down path of length at most~3.}
\end{lemma}

\begin{proof}
 We employ the following distance arguments.
 
 \begin{itemize}
  \item If~$d_N(x,y)=2$ then~$N$ must contain~$\Lambda(x,y)$.
  \item If~$d_N(x,y)=3$ then there is at most one reticulation on the shortest~$xy$ up-down path.
  So if in addition we have that
  \begin{itemize}
   \item there are no reticulations on the shortest~$xy$ up-down path. 
   Then~$N$ must contain~$\lambda(x,y)$ or~$\lambda(y,x)$;
   \item there is one reticulation on the shortest~$xy$ up-down path then~$N$ must contain a reticulated cherry~$K(x,y)$ or~$K(y,x)$. \review{If in addition we have that
   \begin{itemize}
    \item the parent of~$x$ and the parent of~$y$ share a common parent, then we say specifically that~$N$ must contain~$A(x,y)$ or~$A(y,x)$;
    \item the parent of~$x$ and the parent of~$y$ do not share a common parent, then we say specifically that~$N$ must contain~$H(x,y)$ or~$H(y,x)$.
   \end{itemize}}
  \end{itemize}
  \item If~$d_N(x,y)\geq4$ then~$N$ must contain~$\Pi(x,y)$.
 \end{itemize}
\end{proof}
 
\begin{figure}[h!]
 \centering
 \includegraphics[width=\textwidth]{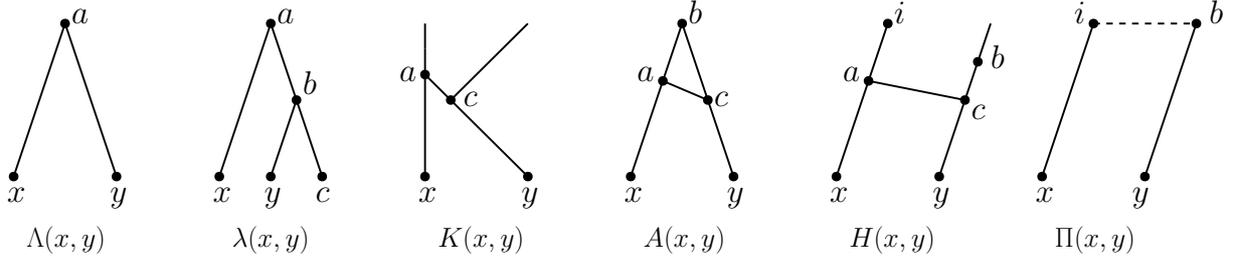}
 \caption{All possible shapes on two leaves~$\{x,y\}$ (up to permuting~$x$ and~$y$).
 The dashed line indicates that any~$ib$ up-down path has length at least~2.}
 \label{fig:Treechildxyshapes}
\end{figure}

We now show that the shape on leaves~$x$ and~$y$ in a tree-child network is identifiable from the shapes on~$x$ and~$y$ in its MLLSs. This is summarized in Table~\ref{tab:allcases}. We start with the following theorem, which shows that each shape is preserved in at least one MLLS.

\begin{theorem}\label{thm:OriginalContainsThenMLLSContains}
 Let~$N$ be a level-$k$ tree-child network where~$k\geq2$, and let~$x,y$ be two leaves in~$N$.
 If~$N$ contains~$\Lambda(x,y)$, $\lambda(x,y)$, $A(x,y)$, $H(x,y)$ or~$\Pi(x,y)$, then there is an MLLS of~$N$ containing~$\Lambda(x,y)$, $\lambda(x,y)$, $A(x,y)$, $H(x,y)$ or~$\Pi(x,y)$ respectively.
\end{theorem}

\begin{proof}
 In this proof, we refer to the node labels used in Lemma~\ref{lem:PossibleShapes}.
 
 The case that~$N$ contains~$\Lambda(x,y)$ is trivial.
 
 Now suppose~$N$ contains~$\lambda(x,y)$.
 If~$c$
 , the sibling of~$y$, is a reticulation then deleting the reticulation edge leading into~$c$ that is not~$(b,c)$ returns an MLLS containing~$\lambda(x,y)$.
 If~$c$ is not a reticulation then deleting any reticulation edge will not affect the shortest~$xy$ up-down path.
 This results in an MLLS containing~$\lambda(x,y)$.\\
 
 
 Suppose~$N$ contains~$A(x,y)$.
 As~$x,y$ are leaves,~$A(x,y)$ is a level-$1$ blob and thus by definition, every MLLS of~$N$ contains~$A(x,y)$.\\
  
 Suppose~$N$ contains~$H(x,y)$.
 If the blob containing the reticulation of~$H(x,y)$ is of level lower than~$k$, then every MLLS of~$N$ contains~$H(x,y)$, and we are done.
 So suppose this blob is level-$k$.
 As~$k\geq2$, there exists a reticulation~$r$, which is not~$c$, with reticulation edges~$e$ and~$f$.
 Let~$N'$ and~$N''$ be the MLLSs of~$N$ obtained by deleting~$e$ and~$f$ (amongst other reticulation edges) respectively.
 We claim that at least one of~$N'$ or~$N''$ contains~$H(x,y)$.
 Indeed, if~$N'$ contains~$A(x,y)$, then in~$N$, either~$b$ or~$i$ must be incident to~$e$, as otherwise~$a$ and~$c$ will still have different parents after deleting~$e$ and cleaning up.
 Now,~$b$ cannot be incident to~$e$ as it violates the tree-child property, regardless of whether~$b$ is the tree node or the reticulation incident to~$e$.
 Then~$i$ must be incident to~$e$.
 If~$i$ is~$r$, then we note that~$b$ cannot be the parent of~$i$ due to the tree-child property.
 This implies that upon deleting~$e$ and cleaning up,~$a$ and~$c$ have different parents, and subsequently~$N'$ contains~$H(x,y)$.
 Thus this case is impossible.
 If, on the other hand,~$i$ is the tree node of~$e$, then neither~$i$ nor~$b$ are suppressed after deleting~$f$ and cleaning up.
 This implies that~$N''$ contains~$H(x,y)$.
\\
 


 Suppose~$N$ contains~$\Pi(x,y)$.
 Suppose first that~$d_N(x,y)\geq 5$.
 Take any~$xy$ up-down path in~$N$, and consider~$BT(N)$.
 Note that any up-down path in~$N$ can be mapped to an up-down path in~$BT(N)$.
 The `up' portion of the path passes through the blob nodes containing~$x$ in their label, until the first blob node containing~$y$ is reached.
 The `down' portion of the path passes through the blob nodes containing~$y$ in their label, until a lowest blob node containing~$y$ is reached.
 In particular, the apex is contained in the lowest blob which contains both~$x$ and~$y$ in their leaf-descendant set.
 So every~$xy$ up-down path in~$N$ passes through the same set of blobs~$\cB$.
 Furthermore, every~$xy$ up-down paths enter and leave the blobs~$B\in\cB$ at the same nodes.
 Let~$t_B$ and~$h_B$ denote these nodes respectively.

 We claim that there is a reticulation edge we can delete from any blob~$B \in \cB$ of level-$k$ such that every~$xy$ up-down path uses at least one edge from~$B$ in the resultant subnetwork.
 We assume~$lvl(B) = k$ as otherwise the claim holds trivially.
 At least one of~$t_B$ or~$h_B$ must be a reticulation, since we enter, pass through, and leave the blob~$B$.
 We consider the cases when they are both reticulations and when~$t_B$ is a reticulation but~$h_B$ is not.
 Suppose first that~$t_B$ and~$h_B$ are both reticulations.
 Then~$B$ must contain the apex of any~$xy$ up-down path; furthermore because of the tree-child property, the shortest~$t_Bh_B$ up-down distance must be at least~3.
 Then deleting a reticulation edge incident to~$h_B$ either disconnects the~$xy$ up-down path or reduces the length by at most~1.
 In any case, at least one edge of~$B$ is still used in the~$xy$ up-down paths in the resultant subnetwork.
 Now suppose that~$t_B$ is the only reticulation.
 Suppose~$h_B$ is not incident to any reticulation edge.
 Since~$lvl(B)=k\geq2$, there exists a reticulation edge we can delete from~$B$, such that neither~$t_B, h_B$, nor the edge~$(t_B,h_B)$ are suppressed.
 Now suppose~$h_B$ is incident to a reticulation edge into a reticulation~$r$.
 If this edge is also incident to~$t_B$, then again since~$lvl(B)=k\geq2$, there exists a reticulation edge we can delete from~$B$, such that neither~$t_B, h_B$, nor the edge~$(t_B,h_B)$ are suppressed.
 Finally if the edge is not incident to~$t_B$, then deleting the reticulation edge incident to~$r$ that is not~$(h_B,r)$ ensures that~$t_B,h_B$, nor~$(t_B,h_B)$ are suppressed.
 In any case, deleting the chosen reticulation edge returns a subnetwork in which an edge of~$B$ is used in every~$xy$ up-down path.


 So if~$|\cB|\geq 2$, then there exists an MLLS~$N^{mlls}$ in which all~$xy$ up-down paths use at least two edges from the blobs in~$\cB$ plus at least three edges connecting the two blobs,~$x$, and~$y$.
 Therefore,~$d_{N^{mlls}}(x,y)\geq 5$.
 If~$|\cB| = 1$ then by Lemma~\ref{lem:EdgeDeletionReducesOnly1Distance}, there exists an MLLS~$N^{mlls}$ with~$d_{N^{mlls}}(x,y)\geq 4$.
 Thus, if~$d_N(x,y)\geq 5$ then there is an MLLS of~$N$ containing~$\Pi(x,y)$ (Figure~\ref{fig:Pi5containsPi}).\\

 \begin{figure}[h]
  \centering
  \includegraphics[scale = 0.4]{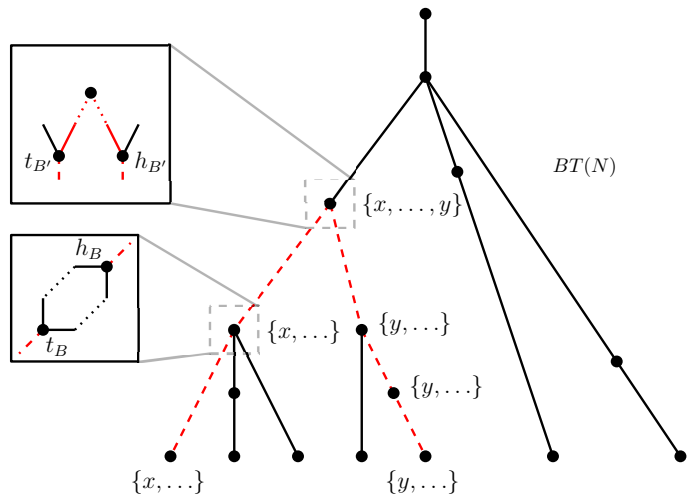}
  \caption{
  Proof visual of Theorem~\ref{thm:OriginalContainsThenMLLSContains},~$d_N(x,y)\geq5$ case.
  The red dashed up-down path in~$BT(N)$ represents the trajectory of every~$xy$ up-down path in~$N$, and consequently the set of blobs~$\cB$ through which every~$xy$ up-down path passes.
  A zoomed-in portion of the two particular blob nodes illustrates the entry point~$t_B$ and exit~$h_B$ in~$N$, and the case for when both points can be reticulations.}
  \label{fig:Pi5containsPi}
 \end{figure}

 Suppose now that~$d_N(x,y)=4$.
 We first show that there are at most 2 shortest~$xy$ up-down paths in~$N$.
 Let~$u,v$ be the parents of~$x,y$ respectively.
 Then any shortest~$xy$ up-down path is always of the form~$(x,u),(u,w),(w,v),(v,y)$ (disregarding directions) where~$w$ is some node in~$N$, and one of~$u,v,w$ is the apex of the shortest up-down path.
 Note that~$u$ and~$v$ are always included in any~$xy$ up-down path, since they are the parents of~$x$ and~$y$ respectively. 
 Therefore, having two shortest~$xy$ up-down paths where~$u$ and~$v$ are the apex in each would create a cycle in~$N$, contradicting the fact that~$N$ is a phylogenetic network.
 Therefore if~$u$ is the apex of a shortest~$xy$ up-down path in~$N$ then there cannot be a shortest~$xy$ up-down path where~$v$ is the apex.
 There can be, however, a second shortest~$xy$ up-down path in~$N$ where~$w$ is the apex. 
 
 Since~$u,v$ are contained in all~$xy$ up-down paths, we have that if two shortest~$xy$ up-down paths have the same apex then they must be the same up-down paths.
 Otherwise the network would not be binary, or there would be parallel edges.
 If there were more than two shortest~$xy$ up-down paths then at least one of~$u$ or~$w$ would have degree greater than~$3$.
 This implies~$N$ is nonbinary, so there can be at most two shortest~$xy$ up-down paths.
 This is shown in Figure~\ref{fig:Distance4IncludesPi}.
 Note that if there are two shortest~$xy$ up-down paths in~$N$ then it must be isomorphic to the one shown in Figure~\ref{fig:Distance4IncludesPi}, as otherwise the only other option would be to have~$w$ and~$w'$ be the apex, in which case~$w'$ would be a parent of 2 reticulations, deeming~$N$ to be not tree-child.
 
 \begin{figure}[h]
  \centering
  \includegraphics[scale=0.45]{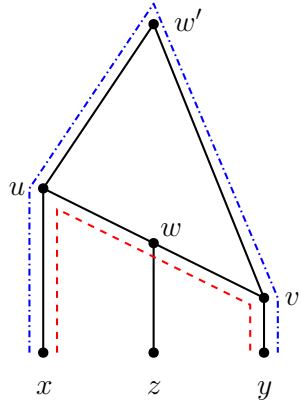}
  \caption{An example of~$2$ shortest~$xy$ up-down paths in~$N$, whenever~$d_N(x,y)=4$. 
  One up-down path (red dashed) is~$(x,u),(u,w),(w,v),(v,y)$ and the other (blue dash-dotted)~$(x,u),(u,w'),(w',v),(v,y)$ (disregarding directions).}
  \label{fig:Distance4IncludesPi}
 \end{figure}
 
 Now we show that if~$d_N(x,y)=4$ then there is always an MLLS of~$N$ containing~$\Pi(x,y)$.
 
 Suppose first that there are two shortest~$xy$ up-down paths.
 Then, as stated before, it is isomorphic to the diagram shown in Figure~\ref{fig:Distance4IncludesPi}.
 There are no reticulation edges incident to either of the shortest paths other than on the reticulation at~$v$. 
 In particular~$(w,z)$ cannot be a reticulation edge because~$N$ is tree-child. 
 As~$lvl(N)\geq2$, there is another reticulation edge~$e$ incident to a reticulation that is not~$v$.
 Indeed, parents of~$x$ and~$y$ remain different and non-adjacent in the MLLS obtained by deleting~$e$.
 This particular MLLS contains~$\Pi(x,y)$.

 Now suppose there is only one shortest~$xy$ up-down path~$P$.
 There are 5 nodes on~$P$ including~$x$ and~$y$, and there are at most two reticulation edges incident to~$P$ and at most one on~$P$ by the tree-child property.
 Since~$lvl(N)\geq2$, there is at least one reticulation edge such that its deletion does not affect~$P$.
 Deleting this reticulation edge and cleaning up ensures that the parents of~$x$ and~$y$ remain different and non-adjacent in the resultant MLLS.
 Therefore, there exists an MLLS of~$N$ which contains~$\Pi(x,y)$.  

 Thus if~$d_N(x,y)=4$ then there exists an MLLS of~$N$ containing~$\Pi(x,y)$.
 Therefore if~$N$ contains~$\Pi(x,y)$, there exists an MLLS of~$N$ containing~$\Pi(x,y)$.\\
\end{proof}


\begin{lemma}\label{lem:IfPithenNoCherryInMLLS}
 For a tree-child network~$N$, if~$N$ contains~$\Pi(x,y)$ then no MLLS of~$N$ contains~$\Lambda(x,y)$.
\end{lemma}
\begin{proof}
 We prove the contrapositive.
 Suppose one of the MLLSs~$N^{mlls}$ of~$N$ contains~$\Lambda(x,y)$.
 Add the deleted reticulation edges back to~$N^{mlls}$.
 Then every node on a shortest~$xy$ up-down path, excluding the apex \review{and the leaves~$x,y$}, is incident to a reticulation edge.
 We first show that these nodes cannot be pure nodes in~$N$.
 
 Suppose for a contradiction that one of these nodes~$p$ is a pure node in~$N$.
 Then~$p$ must be a tree node, and there must exist two disjoint paths from~$p$ to its reticulation child~$r$.
 Without loss of generality, suppose that~$p$ is above~$x$.
 \review{Since~$p$ must be above~$r$ via a path that does not use the edge~$(p,r)$,}
 there exists a node~$z$ that is above~$x$ and below~$p$ such that~$z$ is above~$r$.
 When we delete the reticulation edges again to obtain~$N^{mlls}$, we must delete two edges from the blob with pure node~$p$, which is impossible.
 We have a contradiction.
 
 Now suppose for a contradiction that there are two nodes~$u,v$ on a shortest~$xy$ up-down path in~$N$ excluding the apex.
 By our assumption,~$u$ and~$v$ are contained in a level-$k$ blob.
 By the above claim, neither~$u$ nor~$v$ can be pure nodes in~$N$, and we note that the blob containing~$u$ contains the apex, and the blob containing~$v$ also contains the apex.
 This implies that~$u$ and~$v$ are contained in the same level-$k$ blob.
 To obtain~$N^{mlls}$, only one of~$u$ or~$v$ can be suppressed. 
 In particular,~$(u,v)$ cannot be an edge in~$N$ as otherwise, this blob would be a level-1 blob.
 This implies that~$N^{mlls}$ does not contain~$\Lambda(x,y)$, a contradiction.
 
 Therefore, there can only be one node on a shortest~$xy$ up-down path in~$N$ excluding the apex, and thus~$d_N(x,y)\leq3$.
 Hence~$N$ does not contain~$\Pi(x,y)$.
\end{proof}

\begin{theorem}\label{thm:ShapeIdentifiability}
 Let~$N$ be a level-$k$ tree-child network where~$k\geq2$, and let~$x,y$ be two leaves in~$N$.
 The shape on~$\{x,y\}$ in~$N$ is identifiable from the shapes on~$\{x,y\}$ in the MLLSs.
\end{theorem}
\begin{proof}
 We now prove a series of claims which state that~$N$ contains a certain shape if and only if there are distinct MLLSs of~$N$ containing certain shape(s) on~$\{x,y\}$, and not containing certain other shape(s) on~$\{x,y\}$.\\

\begin{claime}\label{cla:cherry}
 $N$ contains~$\Lambda(x,y)$ if and only if all MLLSs of~$N$ contain~$\Lambda(x,y)$.
\end{claime}\\
 
\begin{claimproof}
 To show necessity, suppose~$N$ contains~$\Lambda(x,y)$ so that~$d_N(x,y) = 2$.
 Since the parent of~$x$ and~$y$ is a tree node, there is no reticulation edge incident to~$\Lambda(x,y)$.
 Then~$\Lambda(x,y)$ is contained in every maximum subnetwork of~$N$, and therefore in every MLLS of~$N$.

 For sufficiency, suppose for a contradiction that all MLLSs of~$N$ on~$X$ contain~$\Lambda(x,y)$, but~$N$ does not.
 If~$N$ contains~$\lambda(x,y), \lambda(y,x), K(x,y), K(y,x),$ or~$\Pi(x,y)$ then, as these are the only possible shapes and their shapes are preserved in some MLLSs by Theorem~\ref{thm:OriginalContainsThenMLLSContains}, we have our required contradiction.
 Thus the claim holds.
 
\end{claimproof}\\
 
 
\begin{claime}
 $N$ contains~$\lambda(x,y)$ if and only if there exists an MLLS of~$N$ containing~$\lambda(x,y)$ and no MLLSs of~$N$ contain~$\lambda(y,x), K(x,y), K(y,x)$ or~$\Pi(x,y)$.
 \end{claime}\\
 
\begin{claimproof}
 To show necessity note that by Theorem~\ref{thm:OriginalContainsThenMLLSContains}, there is an MLLS of~$N$ that contains~$\lambda(x,y)$.
 The only possible reticulation edge incident to~$\lambda(x,y)$ is at~$b$ whenever~$c$ is a reticulation.
 Deleting the edge~$(b,c)$ returns an MLLS containing~$\Lambda(x,y)$, and deleting the reticulation edge incident to~$c$ that is not~$(b,c)$ returns an MLLS containing~$\lambda(x,y)$.
 All other reticulation edges do not intersect~$\lambda(x,y)$ and hence their deletions do not affect~$\lambda(x,y)$.
 Thus an MLLS of~$N$ does not contain~$\lambda(y,x), K(x,y), K(y,x)$ nor~$\Pi(x,y)$.
 The condition is therefore necessary.
  
 To show sufficiency, suppose for a contradiction that the conditions hold but~$N$ does not contain~$\lambda(x,y)$.
 If~$N$ contains~$\Lambda(x,y)$ then by Claim~1, no MLLSs of~$N$ contain~$\lambda(x,y)$, a contradiction.
 If~$N$ contains~$\lambda(y,x), K(x,y), K(y,x),$ or~$\Pi(x,y)$ then, as these are the only possible shapes and their shapes are preserved in some MLLSs by Theorem~\ref{thm:OriginalContainsThenMLLSContains}, we have our required contradiction.
 The condition is necessary and the claim holds.
 
\end{claimproof}\\

Since~$A(x,y)$ is a level-1 blob in~$N$ for two leaves~$x,y\in X$, Claim~3 is trivially true.\\
 
\begin{claime}\label{cla:rcherryA}
 $N$ contains~$A(x,y)$ if and only if all MLLSs of~$N$ contain~$A(x,y)$.
\end{claime}\\
 
\setcounter{claimproofcounter}{3}

When~$N$ contains~$H(x,y)$, let~$B_H$ be the blob containing the reticulation in~$H(x,y)$.\\
\begin{claime}\label{cla:rcherryH}
 \begin{itemize}
  \item $N$ contains~$H(x,y)$ and~$lvl(B_H)=k$ if and only if there exist distinct MLLSs of~$N$ containing~$\Lambda(x,y)$ and~$H(x,y)$, and no MLLSs of~$N$ contain~$K(y,x)$.
  \item $N$ contains~$H(x,y)$ and~$lvl(B_H)<k$ if and only if all MLLSs of~$N$ contain~$H(x,y)$.
 \end{itemize}
\end{claime}~\\

\begin{claimproof}\label{clapro:rcherryH_lambda}
 We first prove the first statement of the claim.
 We first show necessity.
 Isolating~\review{$\langle x,y\rangle$} returns an MLLS of~$N$ containing~$\Lambda(x,y)$. 
 By Theorem~\ref{thm:OriginalContainsThenMLLSContains}, there is an MLLS of~$N$ which contains~$H(x,y)$.
 For the third condition, suppose for a contradiction that some MLLS~$N^{mlls}$ of~$N$ contains~$K(y,x)$.
 Since we have a reticulation on~$y$ in~$H(x,y)$, and because isolating~\review{$\langle x,y\rangle$} returns~$\Lambda(x,y)$,~$N^{mlls}$ must have been obtained by cutting~\review{$\langle x,y\rangle$}.
 But then we have that the node~$b$, the grandparent of~$y$, has only reticulation children in~$N$, contradicting the tree-child property of~$N$.
 We therefore have necessity.

 To show sufficiency, suppose for a contradiction that the conditions hold but~$N$ does not contain~$H(x,y)$.
 If~$N$ contains~$\Lambda(x,y)$ then by Claim~1, no MLLSs of~$N$ contain~$H(x,y)$, a contradiction. 
 If~$N$ contains~$\lambda(x,y)$ or~$\lambda(y,x)$ then no MLLSs of~$N$ contains~$H(x,y)$ by Claim~2, a contradiction.
 If~$N$ contains~$A(x,y)$ then no MLLSs of~$N$ contains~$\Lambda(x,y)$ by Claim~3, a contradiction.
 If~$N$ contains~$K(y,x)$ then the shape is preserved in some MLLS of~$N$ by Theorem~\ref{thm:OriginalContainsThenMLLSContains}, a contradiction.
 Finally if~$N$ contains~$\Pi(x,y)$ then no MLLS of~$N$ contains~$\Lambda(x,y)$ by Lemma~\ref{lem:IfPithenNoCherryInMLLS}, a contradiction.
 As these are the only possibilities, necessity follows.
 The claim holds for~$lvl(B_H) = k$.\\
 
 We now prove the second statement of the claim.
 We first show necessity.
 Now suppose that~$N$ contains~$H(x,y)$ and~$lvl(B_H)<k$.
 Then none of the reticulation edges in~$B_H$ are deleted to obtain any of the MLLSs of~$N$ by definition.
 It follows that all MLLSs of~$N$ contain~$H(x,y)$.
 
 We now show sufficiency.
 Suppose first that every MLLS of~$N$ contains~$H(x,y)$.
 If~$N$ contained a shape that was not~$H(x,y)$, then there exists an MLLS of~$N$ that contains that particular shape by Theorem~\ref{thm:OriginalContainsThenMLLSContains}.
 As this is a contradiction, we have that~$N$ contains~$H(x,y)$.
 To show that~$lvl(B_H)<k$, we note that if this was not the case, i.e., if~$lvl(B_H) = k$, then we have shown above that an MLLS of~$N$ would contain~$\Lambda(x,y)$, which is a contradiction.
 So we must have that~$N$ contains~$H(x,y)$ and that~$lvl(B_H)<k$.

\end{claimproof}\\

\begin{claime}
 $N$ contains~$\Pi(x,y)$ if and only if there exists an MLLS of~$N$ containing~$\Pi(x,y)$ and no MLLSs of~$N$ contain~$\Lambda(x,y)$.
\end{claime}\\
 
\begin{claimproof}
 We first show necessity. 
 There is an MLLS of~$N$ that contains~$\Pi(x,y)$ by Theorem~\ref{thm:OriginalContainsThenMLLSContains}.
 By Lemma~\ref{lem:IfPithenNoCherryInMLLS}, no MLLSs of~$N$ contains~$\Lambda(x,y)$.
  
 To show sufficiency, suppose for a contradiction that the conditions hold, but that~$N$ does not contain~$\Pi(x,y)$.
 If~$N$ contains~$\Lambda(x,y)$ then by Claim~1, every MLLS of~$N$ contains~$\Lambda(x,y)$, a contradiction.
 If~$N$ contains~$\lambda(x,y)$ or~$\lambda(y,x)$ then no MLLSs of~$N$ contain~$\Pi(x,y)$ by Claim~2, a contradiction.
 If~$N$ contains~$A(x,y)$ or~$A(y,x)$ then all MLLSs of~$N$ contains~$A(x,y)$ or~$A(y,x)$ by Claim~3.
 This is a contradiction as no MLLSs of~$N$ would contain~$\Pi(x,y)$.
 If~$N$ contains~$H(x,y)$ or~$H(y,x)$ then we split into two cases.
 Recall that~$B_H$ is the blob of~$N$ which contains~$H(x,y)$ or~$H(y,x)$.
 If~$lvl(B_H)<k$ then all MLLSs of~$N$ contains~$H(x,y)$ or~$H(y,x)$ by Claim~4.
 This is a contradiction as no MLLSs of~$N$ would contain~$\Pi(x,y)$.
 If~$lvl(B_H)=k$ then there exists an MLLS of~$N$ which contains~$\Lambda(x,y)$ by Claim~4, a contradiction.
 The condition is sufficient.
 The claim therefore holds.

\end{claimproof}

\end{proof}


Theorem~\ref{thm:ShapeIdentifiability} is summarized in table~\ref{tab:allcases}.
The table covers all of the different cases, showing which shapes can appear in MLLSs given the shape that the original network contains.
For any two rows in the table, there is some column in which one row has a check and the other a cross.
Thus, we can distinguish between any two cases just by looking at the MLLSs, and so we can determine the structure between~$x$ and~$y$ on~$N$.
Because the given shapes are the only possibilities between two leaves~$x$ and~$y$, the table covers all possible cases.

\begin{table}[h!]
\begin{center}
\setlength\extrarowheight{2.5pt}
\resizebox{\textwidth}{!}{\begin{tabular}{ c|c|c|c|c|c|c|c|c| }
 
~$N$ contains &~$\Lambda(x,y)$ &~$A(x,y)$ &~$A(y,x)$ &~$H(x,y)$ &~$H(y,x)$ &~$\lambda(x,y)$ &~$\lambda(y,x)$ &~$\Pi(x,y)$\\ [1ex]
 \hline
~$\Lambda(x,y)$ & \cmark & \xmark & \xmark & \xmark & \xmark & \xmark & \xmark & \xmark \\ [1ex]
 \hline
~$A(x,y)$ & \xmark & \cmark & \xmark & \xmark & \xmark & \xmark & \xmark & \xmark \\ [1ex]
 \hline
~$H(x,y)$ & \cmark (\xmark) & \xmark (\xmark) & \xmark (\xmark) & \cmark (\cmark) & \xmark (\xmark) & \qmark (\xmark) & \qmark (\xmark) & \qmark (\xmark) \\ [1ex]
 \hline
~$\lambda(x,y)$ & \qmark & \xmark & \xmark & \xmark & \xmark & \cmark & \xmark &\xmark \\ [1ex]
 \hline
~$\Pi(x,y)$ & \xmark & \qmark & \qmark & \qmark & \qmark & \qmark & \qmark & \cmark \\ [1ex]
 \hline\hline
~$A(y,x)$ & \xmark & \xmark & \cmark & \xmark & \xmark & \xmark & \xmark & \xmark \\ [1ex]
 \hline
~$H(y,x)$ & \cmark (\xmark) & \xmark (\xmark) & \xmark (\xmark) & \xmark (\xmark) & \cmark (\cmark) & \qmark (\xmark) & \qmark (\xmark) & \qmark (\xmark) \\ [1ex]
 \hline
~$\lambda(y,x)$ & \qmark & \xmark & \xmark & \xmark & \xmark & \xmark & \cmark &\xmark \\ [1ex]
 
\end{tabular}}
\end{center}
\caption{Given a tree-child network~$N$ contains one of the~$\{x,y\}$ shapes listed in the first column, for each shape listed on the first row, this shape must appear in an MLLS of~$N$ if there is a checkmark (\cmark), this shape cannot appear in an MLLS if there is a cross (\xmark), and a question mark (\qmark) means either could be possible.
In the rows~$H(x,y)$ and~$H(y,x)$, the non-bracketed marks are for the case~$lvl(B_H)=k$ and the bracketed marks are for the case~$lvl(B_H)<k$.}
\label{tab:allcases}
\end{table}
 
\section{Reconstructibility of Tree-Child Networks}\label{sec:TCReconstruction}
 
In this section, we show that the class of \emph{tree-child networks}, excluding trees and level-1 networks with girth at most~4, is MLLS-reconstructible and thus level-reconstructible and subnetwork-reconstructible (where the girth is the length of a smallest cycle in the underlying undirected graph).
A pair of level-1 networks with girth~3 and \review{a triple of level-1 networks} with girth~4 that is not subnetwork-reconstructible is shown in Figure~\ref{fig:Girth3and4}. 

\begin{figure}[h]
 \centering
 \includegraphics[scale = 0.6]{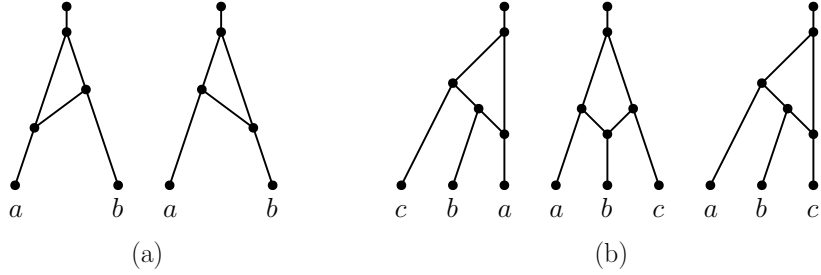}
 \caption{(a) Two non-isomorphic level-1 networks with girth~3 that share the same subnetworks.
 (b) Three non-isomorphic level-1 networks with girth~4 that share the same subnetworks.}
 \label{fig:Girth3and4}
\end{figure}


Following the leaf pair analysis in Section~\ref{sec:LeafPairAnalysis}, we show here that it is possible to infer the location of a missing reticulation edge for level-$k$ blobs from the MLLSs.
By Lemma~\ref{lem:BordewichLemma}, there exists a cherry or a reticulated cherry in every tree-child network.
We know that \review{the common parents within} cherries are level-0 blobs and~$A$ shapes are level-1 blobs.
Then, the reconstruction of level-$k$ blobs can be accomplished by reconstructing an~$H$ shape of every level-$k$ blob.

We start by analyzing the possible shapes on~$x,y$ after cutting a reticulated cherry on~$x$ and~$y$, see Figure~\ref{fig:HSubcases} for examples.

\begin{lemma}\label{lem:HCutShapes}
 Let~$N$ be a tree-child network and suppose~$N$ contains~$H(x,y)$ on a leaf pair~$\{x,y\}$.
 Then the maximum subnetwork obtained by cutting the reticulated cherry~\review{$\langle x,y\rangle$} contains one of~$\lambda(x,y),\lambda(y,x)$, or~$\Pi(x,y)$.
 Furthermore, all other maximum subnetworks of~$N$ contain either~$\Lambda(x,y)$ or~$H(x,y)$.
\end{lemma}

\begin{proof}
 Suppose for a contradiction that cutting~\review{$\langle x,y\rangle$} returns a maximum subnetwork~$N'$ of~$N$ containing either~$\Lambda(x,y), K(x,y),$ or~$K(y,x)$.
 If~$N'$ contains~$\Lambda(x,y)$, then the parent of~$x$ and the parent of~$y$ must share a common parent in~$N$.
 This implies that~$N$ contains~$A(x,y)$, a contradiction.
 If~$N'$ contains~$K(x,y)$, then the parent of~$y$ is a child of a reticulation in~$N$.
 This contradicts the tree-child property of~$N$.
 $N'$ cannot contain~$K(y,x)$ by Theorem~\ref{thm:ShapeIdentifiability}.
 
 To prove the second statement of the lemma, note that isolating~$H(x,y)$ returns a maximum subnetwork of~$N$ that contains~$\Lambda(x,y)$, and deleting any reticulation edge that is not incident to~$y$ returns a maximum subnetwork that contains~$H(x,y)$, since the parent of~$x$ and the parent of~$y$ is not suppressed and they are adjacent.
\end{proof}

\begin{figure}[h]
 \centering
 \includegraphics[scale=0.5]{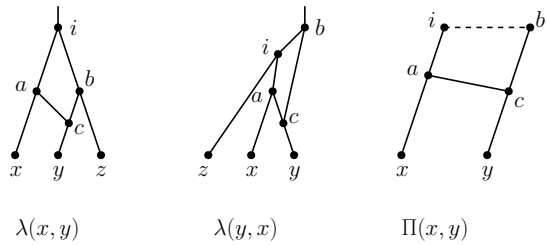}
 \caption{The three cases for~$H(x,y)$ in Lemma~\ref{lem:HCutShapes}. 
 Deleting edge~$(a,c)$ yields~$\lambda(x,y)$,~$\lambda(y,x)$, and~$\Pi(x,y)$ respectively.}
 \label{fig:HSubcases}
\end{figure}

We now show how we can reconstruct a blob containing the reticulation of a reticulated cherry, see Figure~\ref{fig:HReconstruction} for an example.

\begin{lemma}\label{lem:HReconstruct}
 Let~$N$ be a level-$k$ tree-child network, and suppose~$N$ contains~$H(x,y)$ for a leaf pair~$\{x,y\}$.
 Suppose in addition that the blob~$B$ containing the reticulation of~$H(x,y)$ is level-$k$.
 Then we can reconstruct~$B$ in the MLLSs of~$N$.
\end{lemma}

\begin{proof}
 By Theorem~\ref{thm:ShapeIdentifiability} and Lemma~\ref{lem:HCutShapes},~$N$ contains~$H(x,y)$ if and only if all MLLSs of~$N$ contain either~$\Lambda(x,y)$, $H(x,y)$, $\lambda(x,y)$, $\lambda(y,x)$, or~$\Pi(x,y)$, and there exist distinct MLLSs of~$N$ that contain~$\Lambda(x,y), H(x,y)$, and one of~$\lambda(x,y), \lambda(y,x)$, or~$\Pi(x,y)$.
 Now find the MLLS~$N^{mlls}$ of~$N$ that contains one of~$\lambda(x,y), \lambda(y,x)$, or~$\Pi(x,y)$.
 Introduce nodes~$a,c$ directly above~$x,y$ respectively, and add an edge~$(a,c)$ to~$N^{mlls}$.
 This reconstructs the blob~$B$ in~$N^{mlls}$.

It remains to show how to reconstruct~$B$ in the other MLLSs. Hence, consider an arbitrary MLLS~$N^{mlls}_1$. Let~$A\subseteq X$ denote the set of leaf-descendants of the pure node~$p$ of~$B$ in~$N$.
Then $A$ is a node of~$BT(N)$.
 Let~$\Gamma(A)$ denote the set of all \review{children} of~$A$ in~$BT(N)$.
 Let~$p_i$ for~$i=1,\dots,|\Gamma(A)|$ denote the corresponding pure nodes in~$N$.
 In~$N$, delete the tree edge leading into~$p$, and also delete the two outgoing edges of~$p_i$ for~$i=1,\dots,|\Gamma(A)|$, but do not clean up.
 Call the component that contains the node~$p$ the \emph{$B$-part of~$N$}.
 By Theorem~\ref{thm:BTNodeLabelPreservationMLLS}, the MLLS $N^{mlls}_1$ contains \review{a pure node~$q$ with leaf-descendant set~$A$ and pure nodes~$q_i$ with leaf-descendant set equal to each set in~$\Gamma(A)$.}
 \review{In~$N^{mlls}_1$, delete the tree edge leading into~$q$, and also delete the two outgoing edges of~$q_i$, for~$i=1,\dots,|\Gamma(A)|$, but do not clean up.
 Call the component that contains the node~$q$ the \emph{$B$-part of~$N^{mlls}_1$.}}
 We can then reconstruct the blob~$B$ in~$N^{mlls}_1$ by replacing the~$B$-part of~$N^{mlls}_1$ by the~$B$-part of~$N^{mlls}$.
 \review{Since~$B$ is reconstructed correctly in~$N^{mlls}$, and since an edge deletion from a blob does not affect the network outside of the blob, it follows that this replacement correctly reconstructs the blob~$B$ in~$N^{mlls}_1$.} 
\end{proof}

\begin{figure}[h]
 \centering
 \includegraphics[width = \textwidth]{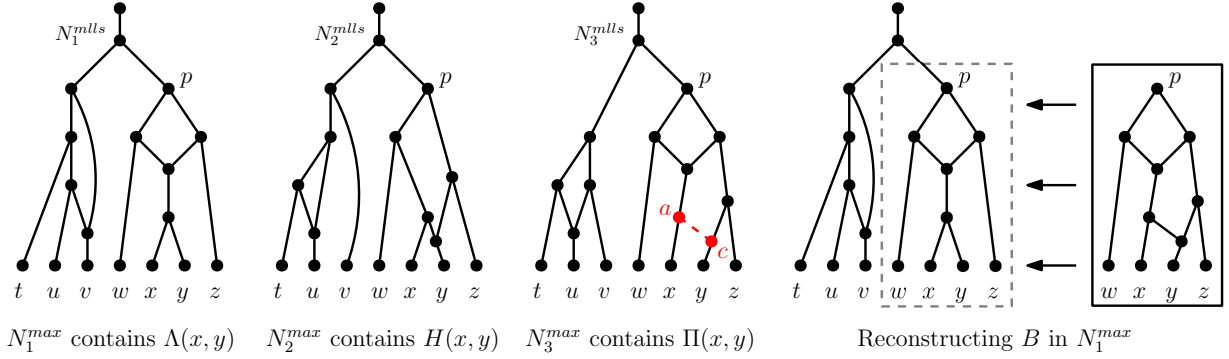}
 \caption{
 Three MLLSs~$N^{mlls}_1, N^{mlls}_2$, and~$N^{mlls}_3$ of a level-2 tree-child network containing exactly two level-2 blobs.
 The three MLLSs contain~$\Lambda(x,y), H(x,y)$, and~$\Pi(x,y)$ respectively.
 We reconstruct the blob~$B$ with pure node~$p$ in~$N^{mlls}_3$ initially, then reconstruct it in the other MLLSs.
 In~$N^{mlls}_3$, nodes~$a,c$ are inserted directly above~$x,y$ respectively, and an edge~$(a,c)$ is added (red dashed edge).
 To reconstruct~$B$ in~$N^{mlls}_1$ the pendant subnetwork rooted at~$p$ is replaced by the reconstructed pendant subnetwork.}
 \label{fig:HReconstruction}
\end{figure}

\begin{definition}
 Let~$N$ be a tree-child network.
 A cherry~$\Lambda(x,y)$ is \emph{reduced} by deleting the node~$y$ and cleaning up (same definition as in~\cite{bordewich2016determining}).
 A reticulated cherry~$K(x,y)$ is \emph{reduced} by isolating~$K(x,y)$ and reducing the resultant cherry~$\Lambda(x,y)$ (different definition to one in~\cite{bordewich2016determining}).
\end{definition}


 

 The following observation shows how we can obtain the MLLSs of a network obtained by reducing a reticulated cherry from the MLLSs of the original network. Note that a maximum subnetwork obtained by isolating a reticulated cherry in a tree-child network remains tree-child by Lemma~\ref{lem:RetEdgeDelNetTC} and that a network obtained by reducing a cherry in a tree-child network also remains tree-child~\cite{bordewich2016determining}.

\begin{lemma}\label{lem:ReconRedux}
 Let~$N$ be a level-$k$ tree-child network with a cherry or a reticulated cherry on \review{a leaf pair}~$\{x,y\}$, and let~$N'$ be the tree-child network obtained by reducing~$\{x,y\}$ from~$N$.
 \begin{itemize}
  \item If~$N$ contains~$H(x,y)$ and the blob~$B$ containing~$H(x,y)$ is of level-$k$, then, in each MLLS of~$N$, reconstruct~$B$ by Lemma~\ref{lem:HReconstruct} and subsequently reduce the reticulated cherry~$H(x,y)$.
  \item Otherwise, reduce~$\{x,y\}$ in all MLLSs of~$N$.
 \end{itemize}
 Let~$\cS$ denote the set of networks we obtain from either of the above two cases.
 Then~$\cS$ is precisely the set of all MLLSs of~$N'$.
\end{lemma}
\begin{proof}
 \review{
 The MLLSs of~$N'$ are obtained by first reducing~$\{x,y\}$ from~$N$ and then finding the MLLSs of the resulting network.
 Since deleting edges from blobs has no effect on all other blobs in the network, we can in fact switch the order of reducing~$\{x,y\}$ and deleting edges from level-$k$ blobs.
 In particular, the MLLSs of~$N'$ can also be obtained by deleting exactly one reticulation edge from all level-$k$ blobs of~$N$ (that are not~$B$, in the case that~$N$ contains~$H(x,y)$ and the blob~$B$ containing~$H(x,y)$ is of level-$k$), and then subsequently reducing~$\{x,y\}$ in all the resulting subnetworks.
 The latter process of obtaining the MLLSs of~$N'$ is exactly how the set~$\cS$ is obtained, and therefore~$\cS$ is precisely the set of all MLLSs of~$N'$.}
\end{proof}

\begin{theorem}\label{thm:TCNetworksReconstructible}
 The class of binary level-$k$ tree-child networks is MLLS-reconstructible, for $k\geq 2$.
\end{theorem}

\begin{proof}
 We prove by induction on~$|X|$ that, for each level-$k$ tree-child network~$N$ on~$X$ with MLLS set~$\cM$, the network~$N$ is the unique level-$k$ tree-child network with MLLS set~$\cM$.
 The base case~$|X|=1$ is trivially true as when there is only one leaf, any network of level-2 or higher is no longer tree-child.
 So suppose~$|X|>1$ and that the claim is true for each level-$k$ tree-child network on at most~$|X|-1$ leaves.
 Let~$N$ be a level-$k$ tree-child network on~$X$ and let~$\cM$ be its MLLS set.
 \review{We will show that the network~$N$ can be reconstructed from~$\cM$, thus showing that~$N$ is the unique network with MLLS set~$\cM$.}
 
 By Lemma~\ref{lem:BordewichLemma}, there exists at least one leaf pair~$\{x,y\}$ that forms a cherry or a reticulated cherry in~$N$.
 
 
 If~$N$ contains~$H(x,y)$ and the blob~$B$ containing the reticulation of~$H(x,y)$ is of level-$k$, then \review{correctly} reconstruct~$B$ in each element of~$\cM$ as outlined in Lemma~\ref{lem:HReconstruct}. 
 \review{By Theorem~\ref{thm:ShapeIdentifiability}, there is no other way of reconstructing the blob~$B$.}
 \review{We update the elements of~$\cM$ by doing so.}
 If~$B$ is the only level-$k$ blob, we are done. Otherwise, we proceed as follows.
Note that we can do this, as we can identify all level-$k$ blobs by Lemma~\ref{lem:blobgraphnew12}.

 \review{At this point, all networks in~$\cM$ contain the same shape on~$\{x,y\}$. Either all networks contain~$\Lambda(x,y)$, all networks contain~$A(x,y)$, or all networks contain~$H(x,y)$.}
 
 Reduce~$\{x,y\}$ in each network of~$\cM$, and call this new set of networks~$\cS$.
  Each network in~$\cS$ is tree-child and contains $|X|-1$ leaves.
 By Lemma~\ref{lem:ReconRedux},~$\cS$ is the set of all MLLSs of~$N'$, the level-$k$ tree-child network obtained by reducing~$\{x,y\}$ in~$N$.
 By the induction hypothesis,~$N'$ is the unique level-$k$ tree-child networks with MLLS set~$\cS$.
 Reconstructing~$N'$ and undoing the reduction operation on~$\{x,y\}$ yields the tree-child network~$N$, which is therefore the unique level-$k$ tree-child network with MLLS set~$\cM$.
\end{proof}
 
 
 

Gambette et al. have shown that level-$1$ networks with girth at least~$5$ are level-reconstructible~\cite{gambette2017challenge}.
The next corollary follows from their results, Observation~\ref{obs:reconstructibility}, 
Theorem~\ref{thm:TCNetworksReconstructible}, and the following observation.

\begin{observation}
Let~$N$ and~$N'$ be two tree-child networks that are both either level at least~2 or girth at least~4. If the level of~$N$ and~$N'$ is different, then they do not have the same set of lower-level subnetworks.
\end{observation}

\begin{corollary}
The class of tree-child networks, excluding trees and level-$1$ networks with girth at most~4, is MLLS-reconstructible, level-reconstructible and subnetwork-reconstructible.
\end{corollary}

\section{Reconstruction Algorithm for Tree-Child Networks}\label{sec:Algorithm}
In this section we present an algorithm in the form of pseudo-code for reconstructing tree-child networks from their MLLSs.
As shown in Section~\ref{sec:TCReconstruction}, we need only to reconstruct the~$H$ shapes contained in level-$k$ blobs.
Algorithm~\ref{alg:downupTCReconstruction} systematically rebuilds every level-$k$ blob from the bottom-up, reducing common pendant subnetworks to leaves on the way.
We give an example of Algorithm~\ref{alg:downupTCReconstruction} in Figure~\ref{fig:Algorithm}. To keep the description of the algorithm concise, it assumes that the input~$\mathcal{T}$ consists of the set of MLLSs of some level-$k$ tree-child network, with~$k\geq 2$. Nevertheless, the algorithm can in principle also be used to decide whether such a network exists or not. If the algorithm returns a network~$N$, then we can check whether $\mathcal{T} = \cN^{mlls}(N)$. If this is not the case, or the algorithm fails to output a network, then such a network does not exist (see Theorem~\ref{thm:alg}). Checking whether $\mathcal{T} = \cN^{mlls}(N)$ can be done in $O(|\mathcal{T}|^2|X|^2)$ time, because checking whether two tree-child networks are isomorphic can be done in~$O(|X|^2)$ time~\cite{cardona2009comparison}.


Moreover, the algorithm can even be applied to an arbitrary set of level-$k-1$ tree-child networks as input. If a network displaying the input networks exists the algorithm may find it, but is not guaranteed to do so (see Theorem~\ref{thm:MinNumMLLS}.)

Before presenting the algorithm, we first go over a few key ideas required to prove the correctness and find the time complexity of the algorithm.
We reiterate the idea of collapsing a pendant subnetwork from a network (presented in Subsection~\ref{subsec:MinNumMLLSBT}, and additionally define what it means to collapse a common pendant subnetwork from a set of networks.
Let~$N$ be a network and let $N_A$ be a pendant subnetwork of~$N$ rooted by a node with leaf-descendant set~$A$.
Collapsing~$N_A$ from~$N$ means that we replace~$N_A$ by a leaf~$A$. 
Let~$N\setminus N_A$ denote the network obtained by collapsing~$N_A$ from~$N$.
Let~$\cM$ be a set of networks containing a common pendant subnetwork~$N_A$.
\emph{Collapsing}~$N_A$ from~$\cM$ means that we collapse~$N_A$ from every network in~$\cM$.
Let~$\cM\setminus N_A$ denote the set of networks obtained by collapsing~$N_A$ from~$\cM$.

\begin{lemma}\label{lem:PendantSubnetwork}
 Let~$N$ be a level-$k$ tree-child network where~$k\geq2$.
 If there exists a common pendant subnetwork~$N_A$ for the MLLSs of~$N$, then~$N_A$ is a pendant subnetwork of~$N$.
\end{lemma}
\begin{proof}
 Consider the blob tree~$BT(N_A)$.
 Since~$N_A$ is a common pendant subnetwork of all MLLSs of~$N$, the blob tree~$BT(N_A)$ is a common pendant subtree of all blob trees of the MLLSs of~$N$.
 By Theorem~\ref{thm:BTNodeLabelPreservationMLLS},~$BT(N)$ must contain~$BT(N_A)$.
 By Lemma~\ref{lem:blobgraphnew12},~$N_A$ must be a level-$k'<k$ network.
 \review{This implies that no edge was deleted from~$N_A$ in obtaining the MLLSs of~$N$.}
 Therefore~$N_A$ is a pendant subnetwork of~$N$.
\end{proof}

The following two observations follow directly from Lemma~\ref{lem:blobgraphnew12}. 

\begin{observation}\label{obs:NoPendantSubnetworkIFFAllLowestLvlK}
 Let~$N$ be a level-$k$ tree-child network where~$k\geq2$.
 There exists no common pendant subnetwork for the MLLSs of~$N$ if and only if all lowest blobs in~$N$ are of level-$k$.
\end{observation}

\begin{observation}\label{obs:PendantSubnetworkBT}
 Let~$N$ be a level-$k$ tree-child network where~$k\geq2$.
 There exists a common pendant subnetwork~$N_A$ for the MLLSs of~$N$ if and only if there exists a common pendant subtree rooted at~$A$ for the blob trees of the MLLSs of~$N$.
\end{observation}

 


\begin{algorithm}[H]
 \KwData{A set $\mathcal{T} = \cN^{mlls}(N)$ for some level-$k$ tree-child network~$N$, where~$k\geq2$}
 \KwResult{The network $N$}
 Update~$\cal{T}$ by collapsing maximal common pendant subnetworks from every network in~$\cal{T}$\;
 Find the blob tree for each network in~$\cal{T}$\;
 Find
 a minimal set~$A$ that is a node of the blob tree of each network in~$\cal{T}$\;
 Find a leaf pair~$\{x,y\}$ where~$x,y\in A$ such that distinct networks~$N_1,N_2,N_3$ of~$\cal{T}$ contain~$\Lambda(x,y), H(x,y)$, and one of~$\lambda(x,y), \lambda(y,x)$, or~$\Pi(x,y)$ respectively\;

 
 
 Update~$N_3$ by adding nodes~$a,c$ directly above~$x,y$ respectively and an edge~$(a,c)$\;
 Let~$N_A$ denote the pendant subnetwork rooted at the top node of this blob in~$N_3$\;
 \For{$N^{mlls}\in\mathcal{T}$}{
 Find the pure node~$p$ with leaf-descendant set~$A$\;
 Replace the pendant subnetwork rooted at~$p$ by~$N_A$\;
 Collapse~$N_A$ from $N^{mlls}$\;
 }
  \uIf{$\cal{T}$ contains a single element~$T$}{
  $N' := T$\;
 }
\Else{$N' := \textsc{TCMLLS-Reconstruction}(\cal{T})$\;
}
 
 Construct~$N$ from~$N'$ by appending the maximal common pendant subnetworks we have collapsed\;
 \Return~$N$

 \caption{Algorithm {\textsc{TCMLLS-Reconstruction}($\cal{T}$)}}
 \label{alg:downupTCReconstruction} 
\end{algorithm}

\begin{figure}[h!]
 \centering
 \includegraphics[scale = 0.45]{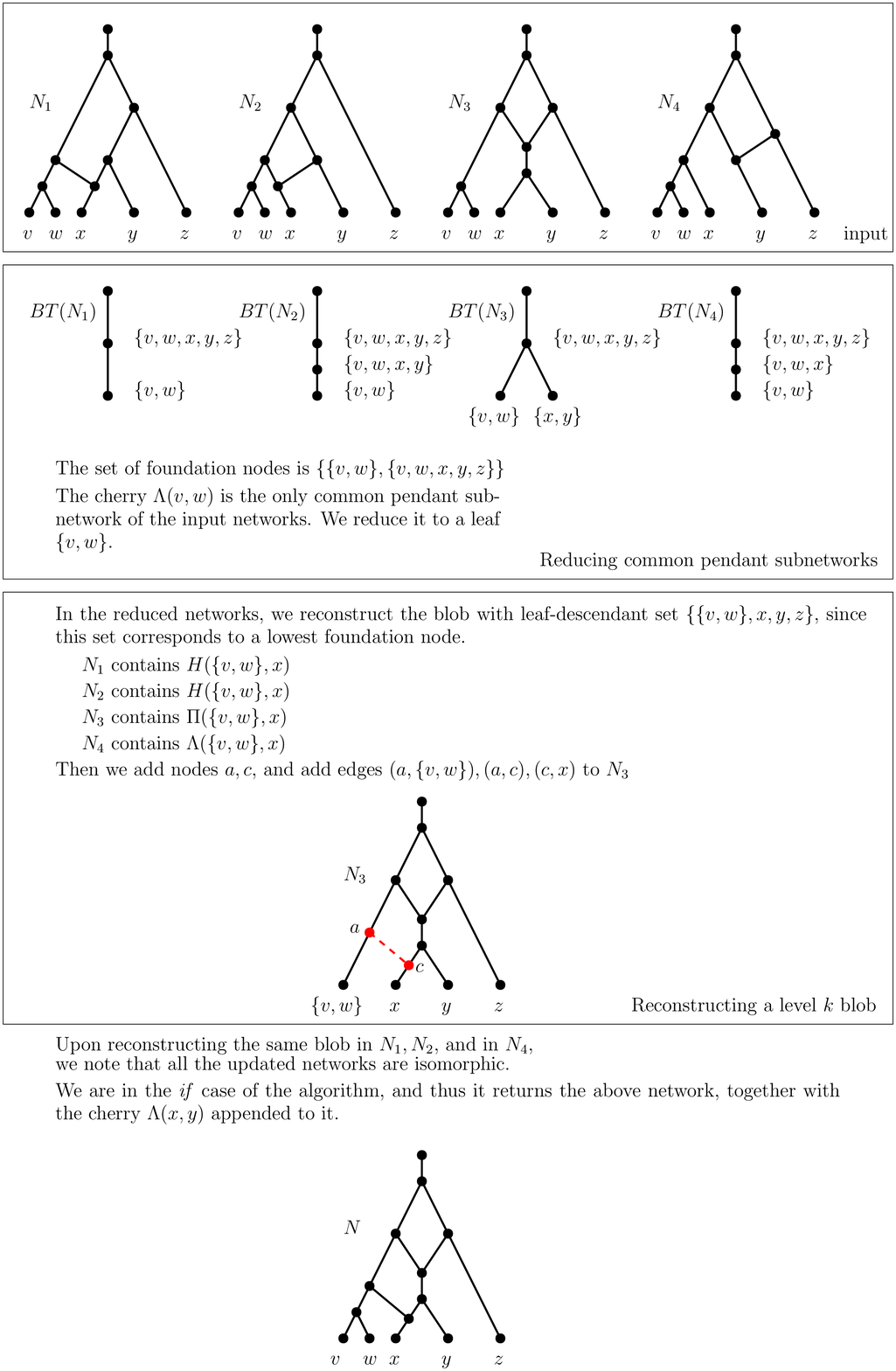}
 \caption{An example of the algorithm \textsc{TCMLLS-Reconstruction}($\{N_1,N_2, N_3, N_4\}$).
 \review{Initially, the common pendant subnetworks of the four input networks are determined by looking at their blob trees.
 In this case, this is the cherry~$\Lambda(v,w)$ (line~1).
 Upon reducing the cherry~$\Lambda(v,w)$ to a leaf~$\{v,w\}$ in all the MLLSs, the lowest foundation node is found to be~$\{\{v,w\},x,y,z\}$.
 We find a leaf pair~$\{\{v,w\},x\}$ specified in line~4 of the algorithm. 
 Since~$N_3$ contains~$\Pi(\{v,w\},x)$, we reconstruct this blob in~$N_3$, shown by the red dashed edge (line~5).
 After reconstructing the same blob in all the other networks~$N_1,N_2$ and~$N_4$, we see that the networks are all isomorphic (we enter the if statement of line~12).
 The algorithm then returns the network~$N$.}}
 \label{fig:Algorithm}
\end{figure}

\begin{theorem}\label{thm:alg}
 Let~$N$ be a level-$k$ tree-child network on~$X$ where~$k\geq2$, and let~$\mathcal{T} = \cN^{mlls}(N)$. 
 Algorithm~\ref{alg:downupTCReconstruction} finds the network~$N$
 in time~$O(|\mathcal{T}||X|^3/k)$.
\end{theorem}

\begin{proof}
 We first prove the correctness of the algorithm.
 By Lemma~\ref{lem:PendantSubnetwork}, every maximal common pendant subnetwork of~$\cal{T}$ is a pendant subnetwork of~$N$.
 Let~$N_A$ be a maximal common pendant subnetwork of~$\cal{T}$.
 Then we can collapse~$N_A$ from~$N$ and~$\cal{T}$, solve a smaller instance of the MLLS reconstruction problem by reconstructing~$N\setminus N_A$ from~$\mathcal{T}\setminus N_A$, and then appending~$N_A$ to the leaf labelled~$A$ -- which is the final step of the algorithm -- returns the network~$N$. 
 Since all maximal pendant subnetworks are disjoint from one another, we can repeat this reduction for all maximal pendant subnetworks, by considering the reductions separately.
 Let~$\cal{T'}$ and~$N'$ denote the set of networks and network obtained by collapsing these pendant subnetworks from~$\cal{T}$ and~$N$ respectively.
 By Observation~\ref{obs:NoPendantSubnetworkIFFAllLowestLvlK}, we have that all lowest blobs of~$N'$ are of level-$k$.
 We search for a lowest level-$k$ blob~$B$ by finding a minimal set~$A'$ that is a node of the blob tree of each network in~$\mathcal{T}$. Then~$A'$ is a lowest foundation node of~$N'$ by Theorem~\ref{thm:BTNodeLabelPreservationMLLS}.
 Then we search for a leaf pair~$\{x,y\}$ which form a reticulated cherry in~$N'$ with the reticulation in~$B$.
 We note that such a leaf pair exists since~$B$ is a lowest blob and since we have collapsed all common pendant subnetworks. Moreover, we can find it by searching for a pair of leaves as described in the algorithm by Table~\ref{tab:allcases} and its proof in Theorem~\ref{thm:ShapeIdentifiability}.
 Now we reconstruct~$B$ using the steps outlined in the proof of Lemma~\ref{lem:HReconstruct}.
 Let~$A'$ denote the leaf-descendant set of~$B$, and let~$N'_{A'}$ denote the corresponding pendant subnetwork (the subnetwork is pendant since~$B$ is a lowest blob).
 At this point, if we only needed to reconstruct one level-$k$ blob (i.e. the case when~$N'$ contains one level-$k$ blob), then we have reconstructed~$N'$.
 Otherwise, collapsing $N'_{A'}$
 gives the full set of MLLSs of the network~$N'\setminus N'_{A'}$.
 Continuing this reasoning, the recursive call will return the network~$N'$.  Appending the collapsed maximum common pendant subnetworks to~$N'$ returns the network~$N$.
 
 Each recursive call of {\sc TCMLLS-Reconstruction} reconstructs one level-$k$ blob, and therefore the algorithm terminates once we have reconstructed all level-$k$ blobs, in which case we have reconstructed the network.\\
 
 For the running time, observe that each recursive call of {\sc TCMLLS-Reconstruction} acts on an instance~$\cN^{mlls}(N')$ on leaf set~$X'$ such that~$|X'|<|X|$ and there is one fewer level-$k$ blob that needs to be reconstructed in the networks of~$\cN^{mlls}(N')$ when compared to that of~$\cN^{mlls}(N)$.
 Since every level-$k$ blob has at least~$k+1$ outgoing edges, the number of level-$k$ blobs in~$N$ is at most~$|X|/k$. 
 Then {\sc TCMLLS-Reconstruction} is called at most~$|X|/k$ times.
 
 Each single iteration of {\sc TCMLLS-Reconstruction} can be split into four parts - collapsing largest common pendant subnetworks from the networks \review{(lines~1-2)}, finding a lowest foundation node~$A$ (i.e., finding a lowest level-$k$ blob) \review{of~$N$ (line~3)}, reconstructing a lowest level-$k$ blob \review{(lines~4-11)}, and checking whether all updated networks are isomorphic to one another \review{(line~12)}.
 Finding a largest common pendant subnetwork can be done by looking at the blob trees, which can all be constructed in~$O(|\mathcal{T}||X|)$ time.
 By Observation~\ref{obs:PendantSubnetworkBT}, there exists a common pendant subnetwork rooted at the pure node with leaf-descendant set~$A$, if there exists a common pendant subtree rooted at~$A$ in the blob trees. 
 The number of blob nodes is maximized for a tree on~$|X|$ leaves (or whenever every reticulation is in a triangle), in which case it contains~$2|X|$ nodes in total (including the root).
 Then there exist at most~$|X|-1$ foundation nodes, and thus at most~$|X|-1$ pendant subnetworks.
 \review{Collapsing a pendant subnetwork from a network takes constant time, and}
 for every possible pendant subnetwork, we iterate through the networks in~$\cal{T}$.
 This step takes~$O(|\mathcal{T}||X|)$ time.
 Finding a lowest foundation node follows immediately as we have found all pendant subnetworks of the blob trees.
Then this step takes constant time.
 The level-$k$ blob reconstruction chooses a pair of leaves which descend from a blob  and subsequently searches through all networks in~$\cal{T}$ to see if the pair is the~$H$ shape we seek. This takes $O(|\mathcal{T}||X|^2)$ time if we try each pair of leaves (or~$O(|\mathcal{T}|^2|X|)$ time if we try each cherry of each network).
 \review{Collapsing pendant subnetworks take constant time, and we do this for every network in~$\cal{T}$. Therefore the running time for reconstructing a lowest level-$k$ blob still takes~$O(|\mathcal{T}||X|^2)$ time (or~$O(|\mathcal{T}|^2|X|)$ time).}
 To decide whether all networks have become isomorphic, we only need to check whether each blob tree consists of just two nodes (including the root).
 It is not necessary to check whether some networks have become isomorphic after each recursion, since the algorithm still works if the input set contains isomorphic networks.
 The algorithm terminates when all level-$k$ blobs have been reconstructed: this is precisely when all networks  become isomorphic.
 Since we can construct (or update) the blob trees in $O(|\mathcal{T}||X|)$ time, we can decide whether all networks have become isomorphic in~$O(|\mathcal{T}||X|)$ time.

 Thus the total time over a single iteration of {\sc TCMLLS-Reconstruction} is~$O(|\mathcal{T}||X|^2)$ (or $O(|\mathcal{T}|^2|X|)$).
 
 It follows that the total running time of the algorithm is~$O(|\mathcal{T}||X|^3/k)$ (or $O(|\mathcal{T}|^2|X|^2/k)$).
\end{proof}

Here we restrict the input data~$\cal{T}$ to be the full-set of MLLSs of some tree-child network~$N$, and return $N$.
We now show that it is not necessary to have this restriction: in fact, we require only three MLLSs to reconstruct~$N$.
That is, the same three MLLSs can be used to reconstruct each level-$k$ blob.
However, if we do not have all MLLSs, we are unable to identify the level-$k$ blobs. Therefore, we require here that also the number of reticulations in the network is given.

\begin{theorem}\label{thm:MinNumMLLS}
 Three MLLSs suffice to reconstruct a level-$k$ tree-child network, with~$k\geq2$, if the number~$l$ of level-$k$ blobs is known.
\end{theorem}
\begin{proof}
 Let~$N$ be a level-$k$ tree-child network with~$l$ level-$k$ blobs and~$k\geq 2$. We first pick three MLLSs $N^{mlls}_1,N^{mlls}_2$, and~$N^{mlls}_3$ of~$N$ and then show that Algorithm~\ref{alg:downupTCReconstruction} returns~$N$ with these inputs and that no other tree-child network exists with these three MLLSs and~$l$ level-$k$ blobs.
 
 Let~$B_1,\dots,B_l$ denote the level-$k$ blobs in~$N$, and let, for~$i = 1,\dots,l$, $r_i$ denote a reticulation in~$B_i$ that is in a lowest reticulated cherry shape, which exists by
 Lemma~\ref{lem:blobseesretcherryshape}.
 Since~$B_i$ is a level-$k$ blob where~$k\geq2$, the parents of~$r_i$ must be non-adjacent (otherwise~$B_i$ would be a level-1 blob).
 Let~$N^{mlls}_1$ denote the MLLS of~$N$ obtained by cutting the reticulated cherry shapes that contain~$r_i$, and let~$N^{mlls}_2$ denote the MLLS of~$N$ obtained by isolating the reticulated cherry shapes that contain~$r_i$.
 Let~$N^{mlls}_3$ denote the MLLS of~$N$ obtained by deleting, from each level-$k$ blob~$B_i$, a reticulation edge that is not incident to~$r_i$, and such that the parents of~$r_i$ remain non-adjacent 
in~$N^{mlls}_3$.
 To see that such an MLLS exists, recall that by Theorem~\ref{thm:OriginalContainsThenMLLSContains}, we have that if a network contains~$H(x,y)$, then there is an MLLS of~$N$ that contains~$H(x,y)$.
 Therefore, by treating each level-$k$ blob~$B_i$ as a level-$k$ tree-child network, we may invoke Theorem~\ref{thm:OriginalContainsThenMLLSContains} to claim that such an MLLS~$N^{mlls}_3$ exists.
 
 Now we show that these three MLLSs suffice to reconstruct~$N$ with Algorithm~\ref{alg:downupTCReconstruction}.
 Recall that Algorithm~\ref{alg:downupTCReconstruction} initially collapses all maximal common pendant subnetworks from the input.
 Let~$N'$ denote the network obtained by collapsing the same pendant subnetworks from~$N$. Then, the algorithm finds a minimal set~$A$ that is a node of the blob tree of each of $N^{mlls}_1, N^{mlls}_2, N^{mlls}_3$. 
 By Lemma~\ref{lem:MinNumberMLLStoReconstructBT},~$A$ is a leaf of the blob tree of~$N'$. 
Since the MLLSs $N^{mlls}_1, N^{mlls}_2, N^{mlls}_3$ were constructed in such a way that all common pendant subnetworks are of level strictly lower than~$k$, the set~$A$ is the set of leaf-descendants of some lowest level-$k$ blob of~$N'$. Note that, in~$N'$,~$r_1$ is contained in a reticulated cherry~$H(x,y)$ for some leaves~$x,y$, since there are no blobs below~$B_1$.
 By construction, the MLLS~$N^{mlls}_1$ contains one of~$\lambda(x,y), \lambda(y,x)$, or~$\Pi(x,y)$.
 The MLLS~$N^{mlls}_2$ contains~$\Lambda(x,y)$, and the MLLS~$N^{mlls}_3$ contains~$H(x,y)$.
 Hence, we can argue similarly to in the proof of Theorem~\ref{thm:alg} that the algorithm then reconstructs the blob~$B_1$ in all input MLLSs, and that recursing the algorithm reconstructs the next lowest level-$k$ blob (which can be reconstructed analogously as done for~$B_1$).
 It follows then that the algorithm reconstructs~$N$ after~$l$ recursions of the algorithm.
 
 We now show that no other tree-child level-$k$ network with~$l$ level-$k$ blobs exists that displays $N^{mlls}_1, N^{mlls}_2$ and $N^{mlls}_3$. Suppose such a network~$N^*$ exists. Then its MLLS set contains $N^{mlls}_1, N^{mlls}_2$ and $N^{mlls}_3$. Hence, by the arguments above, running Algorithm~\ref{alg:downupTCReconstruction} on the full set~$\cN^{mlls}(N^*)$ of MLLSs of~$N^*$ returns~$N$. In particular, note that since~$N^*$ has the same number of level-$k$ blobs as~$N$, the same common pendant subnetworks are collapsed (in each recursive call) when running the algorithm on~$\cN^{mlls}(N^*)$ as when we run the algorithm on $N^{mlls}_1, N^{mlls}_2$ and $N^{mlls}_3$. By Theorem~\ref{thm:alg}, running the algorithm on $\cN^{mlls}(N^*)$ returns~$N^*$. Hence,~$N^*$ and~$N$ are isomorphic.
\end{proof}

Note that in the proof of Theorem~\ref{thm:MinNumMLLS}, we crafted three particular MLLSs~$N^{mlls}_1, N^{mlls}_2,$ and~$N^{mlls}_3$ however, there could be many triples of MLLSs which are sufficient in reconstructing the original network.
Suppose that we have deleted the reticulation edge~$e^j_i$ from the blob~$B_i$ to obtain the MLLS~$N^{mlls}_j$ for~$j=1,2,3$.
The proof of Theorem~\ref{thm:MinNumMLLS} depends on the argument that if for every level-$k$ blob, we can find the particular three shapes, then we can reconstruct said level-$k$ blob.
Then we can define three new MLLSs that are also sufficient for reconstructing~$N$ as follows.
Let~$N_1$ be the MLLS obtained by deleting one of the three reticulation edges ($e^j_1, e^j_2, e^j_3$) from each level-$k$ blob ($B_i$).
Let~$N_2$ be the MLLS obtained by deleting one of the remaining two reticulation edges, and let~$N_3$ denote the MLLS obtained by deleting the remaining reticulation edges from~$N$.
If there were~$l$ level-$k$ blobs in~$N$, then we would have~$3l$ possible choices of~$N_1$,~$2l$ possible choices of~$N_2$, and~1 possible choice of~$N_3$.
Therefore we have~$6l^2$ possible choices for having a triple of MLLSs that are sufficient for reconstructing the network, given the reticulation edges~$e^j_i$.
Note that we simply looked at a particular lowest reticulation~$r_i$ for each level-$k$ blob~$B_i$, and also note that there could be more than one reticulation edge that we could have deleted in retrieving the MLLS~$N^{mlls}_3$ (in the proof of Theorem~\ref{thm:MinNumMLLS}).
This implies that there could be many more triples of MLLSs that suffice to reconstruct~$N$. 



Therefore, it is possible to reconstruct the network from a subset of the MLLSs, given that they hold enough information.
In particular, if we have the three MLLSs as stated in the proof of Theorem~\ref{thm:MinNumMLLS}, then our algorithm returns the network in time~$O(|X|^2/k)$. 

\section{Conclusion and Outlook}\label{sec:Conclusion}

In this paper, we have shown that level-$k$ tree-child networks, where~$k\geq2$, are determined by their MLLSs and provided a polynomial-time algorithm for reconstructing such a network from its MLLSs.
We achieved this result by exploiting one of the tree-child properties - the lowest tree node is in a cherry or a reticulated cherry.

An apparent hindrance to our method is that there is no guarantee nor reason to have the set of all MLLSs of the original network.
Converting sequence data into the MLLSs could be quite challenging, especially for higher level.
It would therefore be interesting to focus on ways to make our results more practical, possibly employing similar approaches used in methods such as trinet-based methods~\cite{oldman2016trilonet}, which work with subnetworks with only three leaves.

On the positive side, we have shown that it is not necessary to know all MLLSs to reconstruct the original network: we need only three, see Theorem~\ref{thm:MinNumMLLS}.
Therefore, a more plausible application of our approach would be the following. Suppose, for example, that different studies each manage to produce a network with some reticulations. However, in each of these networks some actual reticulate events have been missed, possibly due to computational limitations or lack of data. Then, a method based on our theoretical results could be used to reconstruct the full network from the networks with missing reticulate events.



Extending our MLLS reconstructibility results to a more general class of networks is another natural step forward.
We briefly discuss and explain how it might be possible to adapt our results to the class of \emph{valid} networks, where every reticulation edge in the network is valid.


In the case of valid networks, it is not always true that there exists a reticulation edge in every blob \review{whose} deletion results in a maximum subnetwork where the blob tree differs from the blob tree of the original network (i.e., an analogous statement to Observation~\ref{obs:DiffBlobTreeAfterDeletion} does not always hold for valid networks).
An example of this, a tessellating crown, is shown in Figure~\ref{fig:TesselatingCrown}.
\begin{figure}[h!]
 \centering
 \includegraphics[scale=0.5]{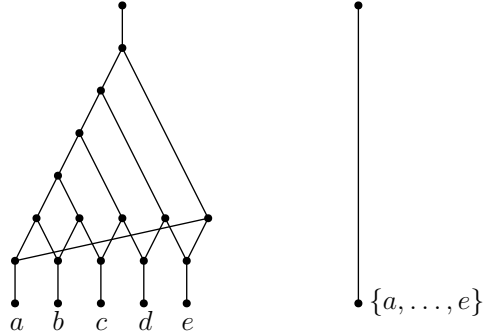}
 \caption{A valid network with a tessellating crown blob and its blob tree.
 Deleting any of the reticulation edges keeps the original blob biconnected, and hence it will not affect the blob tree.}
 \label{fig:TesselatingCrown}
\end{figure}

For leaf pair analysis, there is an extra shape on two leaves~$\{x,y\}$ where both parents~$p_x,p_y$ of~$x,y$ respectively are reticulations, and they share a common parent~$g_x$.
This shape, called a \emph{\review{$2$-reticulated cherry}} (see leaves~$b,c$ of Figure~\ref{fig:TesselatingCrown}, \review{defined in~\cite{Bordewich2018}}), is distinct from all others (given we adapt the definition of when~$N$ contains~$\Pi(x,y)$) as it contains~$K(x,y)$ and~$K(y,x)$ in its MLLSs respectively.
Unlike tree-child networks, MLLSs of valid networks can contain a cherry, reticulated cherry, or a \review{$2$-reticulated cherry} stemming from the lowest tree node.
The reconstruction of a \review{$2$-reticulated cherry} poses a challenge as there are two potential places to reinsert the reticulation edge.
That is, given an MLLS where~$(g_x,p_x)$ has been deleted, we add two nodes~$a,b$ with edge~$(a,b)$. 
We know that~$b$ must be placed directly above~$x$. However, we have two possibilities for inserting~$a$ above~$p_y$ (illustrated in Figure~\ref{fig:DoubleReticulatedCherry}).

\begin{figure}[h]
 \centering
 \includegraphics[scale=0.45]{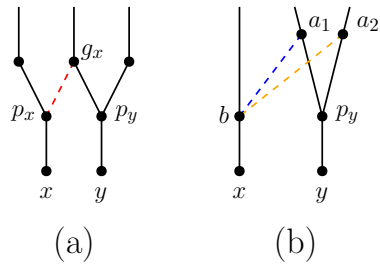}
 \caption{(a) \review{$2$-reticulated cherry} on~$\{x,y\}$.
 (b) MLLS where the red dashed edge is deleted from (a).
 We have two options,~$a_1,a_2$, for inserting the reticulation edge.}
 \label{fig:DoubleReticulatedCherry}
\end{figure}

Nevertheless, we conjecture the following:

\begin{conjecture}
 The class of binary valid networks is MLLS-reconstructible.
\end{conjecture}

Note here that invalid networks, where not every reticulation edge is valid, are not level-reconstructible in general.
A counter example is given in Figure~\ref{fig:InvalidLvl2}.

\begin{figure}[h]
 \centering
 \includegraphics[width = \textwidth]{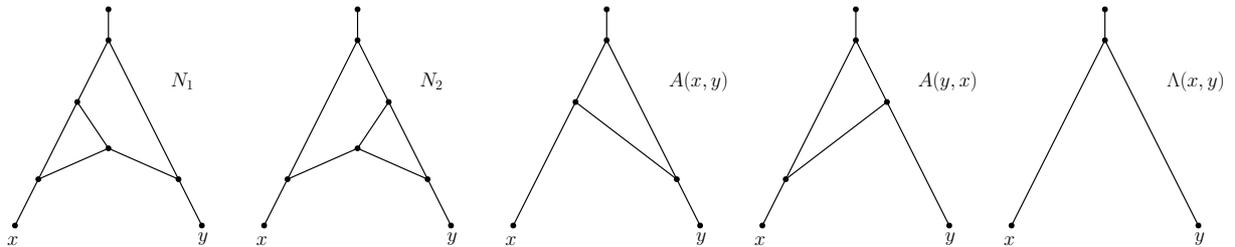}
 \caption{Two level-2 networks~$N_1$ and~$N_2$ are non-isomorphic but have the same lower-level subnetworks.
 In general, invalid networks are not level-reconstructible.}
 \label{fig:InvalidLvl2}
\end{figure}

Ultimately we wish to
characterize precisely which networks are reconstructible from their MLLSs.
Though this could perhaps be possible by an analogous leaf pair analysis as done in Section~\ref{sec:LeafPairAnalysis}, we quickly reach a large number of cases, with level-$2$ networks already containing~15 possible shapes.
A more efficient methodology will be required to treat such general networks.
Considering the level-$k$ generators~\cite{van2009constructing} may perhaps provide an interesting approach to this problem.

\begin{acknowledgements}
We thank the anonymous referee for their helpful comments in improving this manuscript.
Yukihiro Murakami, Leo van Iersel, Remie Janssen, and Mark Jones were supported in part by the Netherlands Organization for Scientific Research (NWO), including Vidi grant 639.072.602, and Leo van Iersel also partly by the 4TU Applied Mathematics Institute.
Vincent Moulton thanks the Netherlands Organization for Scientific Research (NWO) Vidi grant 639.072.602, for its support to visit TU Delft.
\end{acknowledgements}

\bibliographystyle{spmpsci}      
\bibliography{z_bibliography}   


\end{document}